\documentclass{amsart}
\usepackage[normalem]{ulem}
\usepackage[all]{xy}
\usepackage{verbatim}
\usepackage{color}
\usepackage{amsthm}
\usepackage{amssymb}
\usepackage[colorlinks=true]{hyperref}
\numberwithin{equation}{section}

\newtheorem{theorem}[equation]{Theorem}
\newtheorem*{theorem*}{Theorem} \newtheorem{lemma}[equation]{Lemma}

\newtheorem*{conjecture*}{Mamma Conjecture}
\newtheorem*{conjecture1*}{Mamma Conjecture (revisited)}
\newtheorem{proposition}[equation]{Proposition}
\newtheorem{corollary}[equation]{Corollary}
\newtheorem*{corollary*}{Corollary}

\theoremstyle{remark}
\newtheorem{definition}[equation]{Definition}

\newtheorem{example}[equation]{Example}

\newtheorem{notation}[equation]{Notation}

\theoremstyle{remark}
\newtheorem{remark}[equation]{Remark}

\newcommand{\cA}{{\mathcal A}}
\newcommand{\cB}{{\mathcal B}}
\newcommand{\cC}{{\mathcal C}}
\newcommand{\cD}{{\mathcal D}}

\newcommand{\cF}{{\mathcal F}}
\newcommand{\cG}{{\mathcal G}}

\newcommand{\cI}{{\mathcal I}}

\newcommand{\cM}{{\mathcal M}}

\newcommand{\cO}{{\mathcal O}}
\newcommand{\cP}{{\mathcal P}}

\newcommand{\cT}{{\mathcal T}}
\newcommand{\cU}{{\mathcal U}}
\newcommand{\cV}{{\mathcal V}}

\newcommand{\cX}{{\mathcal X}}

\newcommand{\cZ}{{\mathcal Z}}

\newcommand{\Spt}{\mathrm{Spt}}% Spectra

\newcommand{\bbA}{\mathbb{A}}

\newcommand{\bbG}{\mathbb{G}}

\newcommand{\bbN}{\mathbb{N}}
\newcommand{\bbP}{\mathbb{P}}

\newcommand{\bbQ}{\mathbb{Q}}
\newcommand{\bbZ}{\mathbb{Z}}

\def\Sym{\operatorname{Sym}}

\DeclareMathOperator{\id}{id}

\DeclareMathOperator{\Mod}{Mod}
\DeclareMathOperator{\Mot}{Mot}
\DeclareMathOperator{\Fun}{Fun} % Functor category
\newcommand{\dgcat}{\mathrm{dgcat}} % codimension 
\newcommand{\bbK}{I\mspace{-6.mu}K}

\newcommand{\perf}{\mathrm{perf}}

\newcommand{\dg}{\mathrm{dg}}

\newcommand{\uHom}{\underline{\mathrm{Hom}}}
\newcommand{\Hom}{\mathrm{Hom}}
\newcommand{\End}{\mathrm{End}}

\newcommand{\rep}{\mathrm{rep}}

\newcommand{\Ho}{\mathrm{Ho}}

\newcommand{\Hmo}{\mathrm{Hmo}}% Morita homotopy theory
\newcommand{\op}{\mathrm{op}}

\newcommand{\too}{\longrightarrow}

\newcommand{\REnd}{\mathbf{R}\mathrm{End}}

\newcommand{\ie}{\textsl{i.e.}\ }

\def\r{\rightarrow}

\let\oldmarginpar\marginpar
\def\marginpar#1{\oldmarginpar{\tiny #1}}

\begin{document}

\title[The Gysin triangle via localization and $\bbA^1$-homotopy invariance]{The Gysin triangle via \\localization and $\bbA^1$-homotopy invariance}
\author{Gon{\c c}alo~Tabuada and Michel Van den Bergh}

\address{Gon{\c c}alo Tabuada, Department of Mathematics, MIT, Cambridge, MA 02139, USA}
\email{tabuada@math.mit.edu}
\urladdr{http://math.mit.edu/~tabuada}
\thanks{G.~Tabuada was partially supported by a NSF CAREER Award.}
\thanks{M.~Van den Bergh is a senior researcher at the Fund for Scientific Research, Flanders}

\address{Michel Van den Bergh, Departement WNI, Universiteit Hasselt, 3590 Diepenbeek, Belgium}
\email{michel.vandenbergh@uhasselt.be}
\urladdr{http://hardy.uhasselt.be/personal/vdbergh/Members/~michelid.html}

\date{\today}


\abstract{Let $X$ be a smooth scheme, $Z$ a smooth closed subscheme,
  and $U$ the open complement. Given any localizing and
  $\bbA^1$-homotopy invariant of dg categories~$E$, we construct an
  associated Gysin triangle relating the value of~$E$ at the dg
  categories of perfect complexes of $X$, $Z$, and $U$. In the
  particular case where~$E$ is homotopy $K$-theory, this Gysin
  triangle yields a new proof of Quillen's localization theorem, which
  avoids the use of devissage. As a first application, we prove that
  the value of~$E$ at a smooth scheme belongs to the smallest (thick)
  triangulated subcategory generated by the values of~$E$ at the
  smooth projective schemes. As a second application, we
  compute the additive invariants of relative cellular spaces in terms of the bases of the corresponding cells.
Finally, as a third application, we construct
  explicit bridges relating motivic homotopy theory and mixed motives
  on the one side with noncommutative mixed motives on the other
  side. This leads to a comparison between different motivic Gysin
  triangles as well as to an {\'e}tale descent result concerning
  noncommutative mixed motives with rational coefficients.}}

\maketitle
\vskip-\baselineskip
\vskip-\baselineskip
%\vskip-\baselineskip
%\vskip-\baselineskip
%\vskip-\baselineskip
%\tableofcontents

%\bigskip

%\medskip

%-------------------------------------------------------------------------------
\section{Introduction and statement of results}
%-------------------------------------------------------------------------------
A {\em differential graded (=dg) category $\cA$}, over a base field
$k$, is a category enriched over complexes of $k$-vector spaces; see
\S\ref{sub:dg}. Every (dg) $k$-algebra $A$ gives naturally rise to a
dg category with a single object. Another source of examples is
provided by schemes since the category of perfect complexes $\perf(X)$
of every quasi-compact quasi-separated $k$-scheme $X$ admits a
canonical dg enhancement\footnote{When $X$ is quasi-projective this dg
  enhancement is unique; see \cite[Thm.~2.12]{LO}\cite[Thm.~B]{CS}.}
$\perf_\dg(X)$; see Keller
\cite[\S4.6]{ICM-Keller}. %\marginpar{\Michel{Gon\c calo: I prefer to not use implicit hypothesis. So I made the qcqs hypothesis explicit everywhere.}}
%We denote by $\dgcat(k)$ the $1$-category of (essentially small) dg categories. Moreover let $\Hmo(k)$ be obtained from $\dgcat(k)$ by inverting Morita equivalences.
%\marginpar{\Michel{Gon\c calo: I made the definition of $\dgcat$ more precise. Is it ok?}}
Let us denote by $\dgcat(k)$ the category of (essentially small) dg categories, and by $\Hmo(k)$ its localization at the class of Morita equivalences.

A functor $E\colon \dgcat(k) \to \cT$, with values in a triangulated category, is called:
\begin{itemize}
%\item[(C1)] {\em Morita invariant} if it inverts the Morita equivalences (see \S\ref{sub:dg});
\item[(C1)] a {\em localizing invariant} if it inverts the Morita
  equivalences (or equivalently if it factors through the category $\Hmo(k)$)
  and sends short exact sequences of dg categories (see \S\ref{sub:localizing}) to distinguished triangles
\begin{eqnarray*}
0 \too \cA \too \cB \too \cC \too 0 & \mapsto & E(\cA) \too E(\cB) \too E(\cC) \stackrel{\partial}{\too} \Sigma E(\cA)
\end{eqnarray*}
in a way which is functorial for strict morphisms of exact sequences.
% (see Definition \ref{def:localizing})
 \item[(C2)] an {\em $\bbA^1$-homotopy invariant} if it inverts the canonical dg functors $\cA \to \cA[t]$, where $\cA[t]$ stands for the tensor product of $\cA$ and $k[t]$.
\end{itemize}
\begin{example}[Homotopy $K$-theory]\label{ex:KH}
Let $\Ho(\Spt)$ be the homotopy~category~of~spectra. Weibel's homotopy $K$-theory gives rise to a functor $KH\colon \dgcat(k) \to \Ho(\Spt)$ which satisfies conditions (C1)-(C2); see \cite[\S2]{Fundamental}\cite[\S5.3]{A1-homotopy}. When applied to $A$, resp. to $\perf_\dg(X)$, this functor computes the homotopy $K$-theory of $A$, resp. of $X$.
\end{example}
\begin{example}[Nonconnective algebraic $K$-theory with coefficients]
  Let $l$ be a prime. When $l \nmid \mathrm{char}(k)$, mod-$l^\nu$
  nonconnective algebraic $K$-theory gives rise to a functor
  $\bbK(-;\bbZ/l^\nu)\colon\dgcat(k) \to \Ho(\Spt)$ which satisfies
  conditions (C1)-(C2); see \cite[\S1]{Kleinian}. When $l \mid \mathrm{char}(k)$, we can also consider the functor
  $\bbK(-)\otimes \bbZ[1/l]$. When applied to $A$,\ resp. to
  $\perf_\dg(X)$, these functors compute the nonconnective algebraic
  $K$-theory with coefficients of $A$, resp.~of~$X$.
\end{example}
\begin{example}[{\'E}tale $K$-theory]\label{ex:etale}
Let $l$ be an odd prime. Dwyer-Friedlander's {\'e}tale $K$-theory gives rise to a functor $K^{\mathrm{et}}(-;\bbZ/l^\nu)\colon \dgcat(k) \to \Ho(\Spt)$ which satisfies conditions (C1)-(C2); see \cite[\S5.4]{A1-homotopy}. When $l \nmid \mathrm{char}(k)$ and $X$ is moreover regular and of finite type over $\bbZ[1/l]$, $K^{\mathrm{et}}(\perf_\dg(X);\bbZ/l^\nu)$ agrees with the {\'e}tale $K$-theory~of~$X$.
\end{example}
\begin{notation}
Let $E\colon \dgcat(k) \to \cC$ be a functor, with values in an arbitrary category, and $X$ a quasi-compact quasi-separated $k$-scheme. In order to simplify the exposition, we will write $E(X)$ instead of $E(\perf_\dg(X))$.
%To simplify the notations we will make the following convention: for a functor $E:\dgcat(k)\r \cC$ with values in an arbitrary category and a quasi-compact quasi-separated scheme~$X$ we write $E(X):=E(\perf_\dg(X))$.
\end{notation}
\begin{example}[Periodic cyclic homology]
  Let $k$ be a field of characteristic zero and $\cD^\pm(k)$ the
  derived category of $\bbZ/2$-graded $k$-vector
  spaces. Periodic cyclic homology
  gives rise to a functor $HP\colon \dgcat(k) \to \cD^\pm(k)$ which
  satisfies conditions (C1)-(C2); see
  \cite[\S1.5]{Exact}\cite[\S3]{Fundamental}. When applied to $A$,
  resp. to $\perf_\dg(X)$, this functor computes the periodic cyclic
  homology of $A$, resp. of $X$. When $X$ is moreover
  smooth, the classical Hochschild-Kostant-Rosenberg theorem yields the following identifications with de Rham cohomology:
\begin{equation}\label{eq:deRham}
HP^+(X)\simeq \bigoplus_{n\,\mathrm{even}} H^n_{dR}(X) \quad\quad HP^-(X)\simeq \bigoplus_{n\,\mathrm{odd}}  H^n_{dR}(X)\,.
\end{equation}
\end{example}
\begin{example}[Noncommutative motives]\label{ex:Mot}
  Let $\Mot(k)$ be the (closed) symmetric monoidal triangulated
  category of noncommutative motives constructed in
  \cite[\S2]{A1-homotopy}; denoted by
  $\Mot^{\bbA^1}_{\mathrm{loc}}(k)$ in {\em loc. cit.} By
  construction, this category comes equipped with a symmetric monoidal
  functor $\mathrm{U}\colon \dgcat(k) \to \Mot(k)$ which satisfies
  conditions (C1)-(C2). Roughly speaking\footnote{The precise formulation of this universal property makes use of the language of derivators.}, $\mathrm{U}$ is the initial
  functor satisfying these conditions and preserving filtered homotopy
  colimits; %however to make this literally true one has to assume that the functor $E$ in (C1)-(C2) is suitably enhanced.
for further information on noncommutative motives we invite the
  reader to consult the recent book \cite{book}.
\end{example}
Our main result is the following:
\begin{theorem}[Gysin triangle]\label{thm:main}
Let $X$ be a smooth $k$-scheme, $i\colon Z \hookrightarrow X$ a smooth closed subscheme, and $j\colon U \hookrightarrow X$ the open complement of $Z$. For every functor $E\colon \dgcat(k) \to \cT$ which satisfies conditions (C1)-(C2), we have an induced triangle
\begin{equation}\label{eq:triangle}
E(Z) \stackrel{E(i_\ast)}{\too} E(X) \stackrel{E(j^\ast)}{\too} E(U) \stackrel{\partial}{\too} \Sigma E(Z)\,,
\end{equation}
where $i_\ast$, resp. $j^\ast$, stands for the push-forward, resp. pull-back, dg functor.
\end{theorem}
\begin{remark}[Generalizations]
Theorem \ref{thm:main} admits the following generalizations:
\begin{itemize}
\item[(G1)] We may replace the schemes $X$, $Z$, and $U$ by algebraic spaces; consult \S\ref{sub:G1}.% for details.
\item[(G2)] Given a dg category $\cA$, we may replace the dg
  categories $\perf_\dg(Z)$, $\perf_\dg(X)$, $\perf_\dg(U)$, by
  their tensor product with $\cA$.  In the case where
  $\cA=\perf_\dg(Y)$, with $Y$ a quasi-compact quasi-separated $k$-scheme, this corresponds to
 replacing the schemes $X$, $Z$, $U$ by their product
  with $Y$ over $k$;~consult~Lemma~ \ref{lem:product}.
\end{itemize}
\end{remark}
Let $\perf_\dg(X)_Z\subset \perf_\dg(X)$ be the full dg subcategory of
those perfect complexes of $\cO_X$-modules that are supported on
$Z$.~The~bulk~of~the~proof~of~Theorem~\ref{thm:main} consists in
showing that the morphism $E(i_\ast)\colon E(Z) \to
E(\perf_\dg(X)_Z)$ is invertible; see Theorem
\ref{thm:devissage}. This result, which is of independent interest,
should be considered as a new ``d\'evissage'' theorem. Its proof is
based on the description of the dg category $\perf_\dg(X)_Z$ in terms
of a formal dg $k$-algebra (when $X$ is affine) and on a
Zariski\footnote{When $X$ is an algebraic space $\cX$ we use instead a
  Nisnevich descent argument.} descent argument.

Let us now illustrate the general Theorem \ref{thm:main} in some particular cases.
\begin{example}[Fundamental theorem]\label{ex:fundamental}
When $X$ is the affine line $\mathrm{Spec}(k[t])$, $Z$ is the closed point $t=0$, and $U$ is the punctured affine line $\mathrm{Spec}(k[t,t^{-1}])$, the general Gysin triangle \eqref{eq:triangle} reduces to the following distinguished triangle
\begin{equation}\label{eq:triangle-fund}
E(k) \stackrel{E(i_\ast)}{\too} E(k) \stackrel{E(j^\ast)}{\too} E(k[t,t^{-1}]) \stackrel{\partial}{\too} \Sigma E(k)\,.
\end{equation} 
In this case we have $E(i_\ast)=0$; see
\cite[Lem.~4.2]{Fundamental}. Consequently, \eqref{eq:triangle-fund}
gives rise to an isomorphism $E(k[t,t^{-1}])\simeq E(k) \oplus \Sigma
E(k)$. The generalization (G2) yields an isomorphism $
E(\cA[t,t^{-1}])\simeq E(\cA)\oplus\Sigma
E(\cA)$ for every dg category $\cA$.
By taking $E=KH$,
resp. $E=HP$, we hence recover the fundamental theorems in homotopy
$K$-theory, resp. in periodic cyclic homology, established by Weibel
in \cite[Thms.~1.2 and 6.11]{Weibel}, resp. by Kassel in
\cite[Cor.~3.12]{Kassel}; consult \cite{Fundamental} for further
details.
\end{example}
\begin{example}[Quillen's localization theorem]
Homotopy $K$-theory agrees with Quillen's algebraic $K$-theory on smooth schemes. Therefore, when $E=KH$ the general Gysin triangle \eqref{eq:triangle} reduces to the localization theorem
\begin{equation}\label{eq:localization}
K(Z) \stackrel{K(i_\ast)}\too K(X) \stackrel{K(j^\ast)}{\too} K(U) \stackrel{\partial}{\too} \Sigma K(Z)
\end{equation}
established by Quillen in \cite[Chap.~7 \S3]{Quillen}. Quillen's
%\marginpar{\Michel{Gon\c calo: I deleted the description of our method   since we already said almost the same thing above. Also I feel we should not use words that are too strong like ``radically'' since it may offend people.}}  
proof is based on the
d\'evissage theorem for abelian categories and on the equivalence
between $K$-theory and $G$-theory for smooth schemes. As mentioned
above, our proof is different! Moreover, following the generalization (G1), it applies also to algebraic spaces.
\end{example}
%\begin{example}[Equivariant algebraic $K$-theory]
%Let $G$ be a finite group, $\cX$ a smooth algebraic space equipped with a $G$-action, $i\colon \cZ \hookrightarrow \cX$ a $G$-equivariant smooth closed algebraic space, and $\cU$ the open complement of $\cU$. When $E=KH$, the generalization (G1) applied to the algebraic spaces $\cX/G$, $\cZ/G$, and $\cU/G$ (Deligne; see \cite{Kn}) yields a triangle in $G$-equivariant homotopy $K$-theory
%$$
%KH^G(\cZ) \stackrel{KH^G(i_\ast)}\too KH^G(\cX) \stackrel{KH^G(j^\ast)}{\too} KH^G(\cU) \stackrel{\partial}{\too} \Sigma KH^G(\cZ)\,;
%$$
%consult \cite{Kr,Thomason} for details on equivariant algebraic $K$-theory.
%\end{example}
\begin{example}[Six-term exact sequence]
The maps $i\colon Z \hookrightarrow X$ and $j\colon U \hookrightarrow X$ give rise to homomorphisms on de Rham cohomology $H^n_{dR}(i_\ast)\colon H^n_{dR}(Z) \to H^{n+2c}_{dR}(X) $ and $H^n_{dR}(j^\ast)\colon H^n_{dR}(X) \to H^n_{dR}(U)$ where $c:=\mathrm{codim}(i)$.
Therefore, when $E=HP$ the long exact sequence associated to the general Gysin triangle \eqref{eq:triangle} reduces, via the identification \eqref{eq:deRham}, to the following six-term exact sequence:
$$
\xymatrix{
\bigoplus_{n\,\mathrm{even}}H^n_{dR}(Z) \ar[rr]^-{\bigoplus_nH^n_{dR}(i_\ast)} && \bigoplus_{n\,\mathrm{even}}H^n_{dR}(X) \ar[rr]^-{\bigoplus_nH^n_{dR}(j^\ast)} && \bigoplus_{n\,\,\mathrm{even}}H^n_{dR}(U) \ar[d]^-\partial \\
\bigoplus_{n\,\mathrm{odd}}H^n_{dR}(U) \ar[u]^-\partial && \bigoplus_{n\,\mathrm{odd}}H^n_{dR}(X) \ar[ll]^-{\bigoplus_n H^n_{dR}(j^\ast)} && \bigoplus_{n\,\,\mathrm{odd}}H^n_{dR}(Z) \ar[ll]^-{\bigoplus_nH^n_{dR}(i_\ast)} \,.  
}
$$ %\marginpar{\Michel{Gon\c calo: We should probably add a ``one may check'' here. Otherwise in principle we should check that the    morphisms on the De Rham cohomology level are the same as on  the cyclic homology level (unless you have reference for this). Since the bridge between $HP(\perf)$ and $H_{dR}$ is a bit involved this checking would require some work,  I think.}}
One may check that this sequence is the ``$2$-periodization'' of the
Gysin long exact sequence on de
Rham~cohomology~constructed~by~Hartshorne~in~\cite[Chap.~II~\S3]{Hartshorne}.  
\end{example}
\begin{example}[Noncommutative motivic Gysin triangle]
When $E=\mathrm{U}$ the general Gysin triangle \eqref{eq:triangle} reduces to the {\em noncommutative motivic Gysin triangle}:
\begin{equation}\label{eq:Gysin-mot1}
\mathrm{U}(Z) \stackrel{\mathrm{U}(i_\ast)}{\too} \mathrm{U}(X) \stackrel{\mathrm{U}(j^\ast)}{\too} \mathrm{U}(U) \stackrel{\partial}{\too} \Sigma \mathrm{U}(Z)\,.
\end{equation}
Consult Remarks \ref{rk:motivic1} and \ref{rk:motivic2} for the relation between \eqref{eq:Gysin-mot1} and the motivic Gysin triangle(s) constructed by (Morel-)Voevodsky.
%As proved in \cite[Thm.~2.4]{A1-homotopy}, given a dg category $\cA$, we have $\Hom_\Spt(\mathrm{U}(k),\mathrm{U}(\cA))\simeq KH(\cA)$, where $\Hom_\Spt(-,-)$ stands for the spectral enrichment of $\Mot(k)$. Therefore, Quillen's localization theorem \eqref{eq:localization} can be recovered from \eqref{eq:Gysin-mot1} by applying the functor $\Hom_\Spt(\mathrm{U}(k),-)$.
\end{example}
We conclude this section with the following remark:
\begin{remark}
  Theorem \ref{thm:main} is {\em false} if we assume (C1) but {\em not} (C2). For example, cyclic homology gives rise to a localizing invariant
  $HC\colon \dgcat(k) \to \cD(k)$ which is {\em not} $\bbA^1$-homotopy invariant; see
  \cite[\S5.3]{ICM-Keller}\cite[\S1.5]{Exact}. Following Kassel
  \cite[\S3.4]{Kassel}, we have the following computation
\begin{equation*}
HC_n(k[t,t^{-1}]) \simeq  \left\{
  \begin{array}{lr}
    HC_n(k)\oplus HC_{n-1}(k) \oplus k \oplus I & n=0 \\
    HC_n(k)\oplus HC_{n-1}(k) \oplus I  & n\neq 0\,,
  \end{array}
\right.
\end{equation*}
where $I$ stands for the augmentation ideal of $k[t,t^{-1}]$. Therefore, we conclude from Example \ref{ex:fundamental} that Theorem \ref{thm:main} is false when $E=HC$.
\end{remark}
%-------------------------------------------------------------------------------
\section{Applications}
%-------------------------------------------------------------------------------
%-------------------------------------------------------------------------------
\subsection{Reduction to smooth projective schemes}
%-------------------------------------------------------------------------------
%We prove the following result. \marginpar{\Michel{Gon\c calo: I made the formulation of this theorem more compact, so that it is easier to read.}}
\begin{theorem}\label{thm:main2}
Let $k$ be a perfect field of characteristic $p\ge 0$ and $E\colon \dgcat(k) \to \cT$ a functor which satisfies conditions (C1)-(C2). Let us write $\cT^{sp}$ for the smallest triangulated subcategory of $\cT$ containing the objects $E(Y)$, with $Y$ a smooth projective $k$-scheme, and $\overline{\cT^{sp}}$ for the thick closure of $\cT^{sp}$ inside $\cT$. Given a smooth $k$-scheme $X$, the following holds:
%Assume that $k$ is a perfect field of characteristic $p\ge 0$  and let $E\colon \dgcat(k) \to \cT$ be a functor which satisfies conditions (C1)-(C2).  Let $\cT^{sp}$ be the smallest trangulated subcategory of $\cT$ containing the objects $E(Y)$ with $Y$ a smooth projective  $k$-scheme and let $\overline{\cT}^{sp}$ be its tick closure. Let $X$ be a smooth $k$-scheme. Then
\begin{itemize}
\item[(i)] When $p=0$, the object $E(X)$ belongs to $\cT^{sp}$;
\item[(ii)] When $p>0$ and $\cT$ is $\bbZ[1/p]$-linear, the object $E(X)$ belongs to $\overline{\cT^{sp}}$.
\end{itemize}
\end{theorem}
\begin{remark}
The proof of item (ii) makes use of three ingredients\footnote{The proof of item (i) makes use of ingredient (a) and of resolution of singularities.}: (a) the Gysin triangles provided by Theorem \ref{thm:main}; (b) Gabber's refined version of de Jong's theory of alterations; (c) a ``globalization'' argument which allows us to pass from $\bbZ_{(l)}$-linearity for all $l\neq p$ to $\bbZ[1/p$]-linearity. Making use of ingredients (a)-(b), and of different ``globalization'' arguments, Bondarko \cite[Thm.~2.2.1]{Bondarko} and Kelly \cite[Prop.~5.5.3]{Kelly} established an analogue of item (ii) in the particular case where $\cT$ is the Voevodsky's triangulated category of (effective) geometric motives. The ``globalization'' argument of Bondarko \cite[Rk.~2.2.2]{Bondarko} (resp. Kelly \cite[Appendix A.2]{Kelly}) assumes the existence of a weight structure (resp. that the triangulated category $\cT$ is compactly generated) and that the objects associated to smooth $k$-schemes are compact. Our ``globalization'' argument avoids all these assumptions\footnote{These assumptions are not known to hold in the case of the triangulated category of noncommutative motives $\Mot(k)$. The fact that the noncommutative motive $\mathrm{U}(X)$ of a smooth $k$-scheme $X$ is a compact object of $\Mot(k)$ is now a consequence of Theorem \ref{thm:main2}.}! In particular, it yields an alternative proof of the results of Bondarko and Kelly.
\end{remark}
\begin{corollary}\label{cor:main2} 
Let $E:\dgcat(k)\r \cT$ be a functor as in Theorem \ref{thm:main2}. Assume furthermore that $\cT$ is well generated (see \cite[Def.~1.15]{Neeman1}), symmetric monoidal\footnote{We assume throughout the article that a symmetric monoidal structure on a triangulated category is
  exact in both variables.}, and that the tensor product $-\otimes-$ preserves  arbitrary direct sums in both variables. Under these assumptions, if the functor~$E$ is moreover symmetric monoidal, then the objects $E(X)$, with $X$ a smooth $k$-scheme, are strongly dualizable.
%Let $E:\dgcat(k)\r \cT$ be as in Theorem \ref{thm:main2} and assume furthermore that $\cT$ is well generated (see \cite[Def.~1.15]{Neeman1}) and symmetric 
%monoidal\footnote{We always assume that a symmetric monoidal structure on a triangulated category is
 % exact in both variables.} with a tensor product that preserves  arbitrary direct sums in both variables.
%If the functor~$E$ is symmetric monoidal, then the objects $E(X)$ for $X$ a smooth $k$-scheme are strongly dualizable.
\end{corollary}
\begin{proof}
Given an object $b \in \cT$, the functor $-\otimes b \colon \cT \to
  \cT$ is exact and preserves arbitrary direct
  sums. Therefore, thanks to \cite[Thm.~8.4.4]{Neeman1}, it admits a
  right adjoint $\uHom(b,-)$ which by definition is the internal-Hom
  functor. This shows that the symmetric monoidal structure of $\cT$
   is closed.

  As proved in \cite[Thm.~1.43]{book}, the strongly dualizable objects 
  of the category $\Hmo(k)$ are the smooth proper dg
  categories; see \S\ref{sub:dg}. Since by assumption the functor~$E$
  is symmetric monoidal, we  conclude that the objects
  $E(X)$, with $X$ a smooth projective $k$-scheme,
  are strongly dualizable. The result follows now from Theorem \ref{thm:main2} and from the well-known
  fact that the strongly dualizable objects of a closed symmetric
  monoidal triangulated category are stable under distinguished
  triangles and direct summands.
\end{proof}
%-------------------------------------------------------------------------------
\subsection{Additive invariants of relative cellular spaces}\label{sub:cellular}
%-------------------------------------------------------------------------------
\subsubsection{Additive invariants}
\label{sec:additive}
Our next application is for so-called \emph{additive invariants} which are a weaker
type of invariant than the kind  we have considered up to now. 
Every localizing invariant is an additive invariant but the converse is not true.\footnote{Quillen's
  algebraic $K$-theory as well as Karoubi-Villamayor's $K$-theory are
  examples of additive invariants which are not localizing; consult
  \cite{book} for details.} Examples of additive invariants
include algebraic $K$-theory and all its variants, cyclic homology and all its variants,
topological~Hochschild~homology,~etc.

\medskip

Let $\mathrm{I}$ be the dg category with objects $\{1,2\}$ and
complexes of morphisms
$\mathrm{I}(1,1)=\mathrm{I}(2,1)=\mathrm{I}(1,2)=k$ and
$\mathrm{I}(2,1)=0$. Given a dg category $\cA$, let
$T(\cA):=\cA\otimes \mathrm{I}$. We have two inclusion dg functors
$\iota_1, \iota_2\colon \cA \to T(\cA)$. A functor $F\colon \dgcat(k)
\to \mathrm{A}$, with values in an additive category, is called an
{\em additive invariant} if it it inverts the Morita equivalences and
sends the dg categories $T(\cA)$ to direct sums
$$ [F(\iota_1)\,\, F(\iota_2)]\colon F(\cA) \oplus F(\cA) \stackrel{\simeq}{\too} F(T(\cA))\,.$$

As explained in \cite[\S13]{Higher}, the notion of additive invariant
can be equivalently formulated in terms of {\em split} short exact
sequences of dg categories. Therefore, every localizing invariant is
in particular an additive invariant. 

\begin{remark} 
As proved in \cite{Additive} (consult also \S\ref{sec:universal}), there exists a
\emph{universal additive invariant}
$\mathrm{U}_{\mathrm{add}}:\dgcat(k)\r \Hmo_0(k)$ with values in a suitable
additive category $\Hmo_0(k)$. This implies that an additive invariant can be alternatively characterized as a
functor $F\colon \dgcat(k)\r \mathrm{A}$ which factors through
$\mathrm{U}_{\mathrm{add}}$.
\end{remark}
\subsubsection{Relative cellular spaces}
A flat map of $k$-schemes $f\colon X \to Y$ is called an {\em affine fibration} of relative dimension $d$ if for every point $y\in Y$ there exists a Zariski open neighborhood $y \in V$ such that $X_V:=f^{-1}(V) \simeq Y \times \bbA^d$ with $f_V\colon X_V \to Y$ isomorphic to the projection onto the first factor. Following Karpenko \cite[Def.~6.1]{Karpenko}, a smooth projective $k$-scheme $X$ is called a {\em relative cellular space} if there exists a filtration by closed subschemes
\begin{equation}\label{eq:filtration}
\varnothing = X_{-1} \hookrightarrow X_0 \hookrightarrow \cdots \hookrightarrow X_i \hookrightarrow \cdots \hookrightarrow X_{n-1} \hookrightarrow X_n=X
\end{equation}
and affine fibrations $p_i\colon X_i \backslash X_{i-1} \to Y_i, 0 \leq i \leq n$, of relative dimension $d_i$ with $Y_i$ a smooth projective $k$-scheme. The smooth schemes $X_i \backslash X_{i-1}$ are called the {\em cells} and the smooth projective schemes $Y_i$ the {\em bases of the cells}.
\begin{example}[$\bbG_m$-schemes]
\label{example:bb}
The celebrated Bialynicki-Birula decomposition \cite{BB} provides a relative cellular space structure on smooth projective $k$-schemes equipped with a $\bbG_m$-action
in which the bases of the cells are given by the connected components of the fixed point locus; consult also \cite[Thm.~3.1]{Brosnan}\cite{H,I}.
This class of relative cellular spaces includes  the isotropic flag varieties considered originally by Karpenko \cite{Karpenko} as well as the isotropic homogeneous spaces  considered later by Chernousov-Gille-Merkurjev ~\cite{CGM}. %See \cite{Brosnan} for more details.
\end{example}
Our main result concerning relative cellular spaces is the following:
\begin{theorem}\label{thm:cellular2} 
%\marginpar{\Michel{Gon\c calo: I changed the notation from $E$ to $F$ to make it more clear that we do not assume (C1)-(C2) here.}}
Let $X$ be a relative cellular space. For every additive invariant~$F$, we have an induced isomorphism 
$
F(X)\simeq \bigoplus_{i=0}^n F(Y_i)
$.
\end{theorem} 
\begin{remark}[Strategy of the proof]
In order to prove Theorem \ref{thm:cellular2} we consider first the special case $F=\mathrm{U}$, \ie we establish first an induced isomorphism
%The prove this result we first consider the  specal case $F=\mathrm{U}$. That is we prove
\begin{equation}\label{eq:cellular2}
\mathrm{U}(X)\simeq \bigoplus_{i=0}^n \mathrm{U}(Y_i)
\end{equation}
in the category of noncommutative motives $\Mot(k)$.
This decomposition is analogous to a similar result for Chow motives  proved by Karpenko \cite{Karpenko} using refined
properties of Chow and $K$-cohomology
groups. 
%\marginpar{\Michel{Gon\c calo: I have not seen Karpenko's  proof. But the proof by Sebastian Del Bano is similar to ours. So I think saying that our proof is radically  different may be viewed as an exageration.}}  
The proof of
\eqref{eq:cellular2} uses in an essential way the fact that $\mathrm{U}$
satisfies conditions (C1)-(C2) (note that we do \emph{not} require this for $F$ in Theorem \ref{thm:cellular2}!). It is based on the
invariance of noncommutative motives under affine fibrations and on
the observation that the Gysin triangles associated to the
filtration~\eqref{eq:filtration}~are actually split! We cannot immediately obtain Theorem \ref{thm:cellular2} from \eqref{eq:cellular2} since $F$ will in general not factor
through $\mathrm{U}$. However, using the fact that all the schemes in the motivic decomposition \eqref{eq:cellular2} are smooth projective, we prove that \eqref{eq:cellular2}
implies a similar decomposition
\[
\mathrm{U}_{\mathrm{add}}(X)\simeq \bigoplus_{i=0}^n \mathrm{U}_{\mathrm{add}}(Y_i)
\]
in the additive category $\Hmo_0(k)$. To finish the proof we use the fact that $F$, being additive, factors (uniquely) through $\mathrm{U}_{\mathrm{add}}$.
\end{remark}
\begin{example}[Kn\"orrer periodicity] 
  The following application of the Bialynicki-Birula decomposition was
  inspired by the work of Brosnan \cite{Brosnan}. Let $q=fg+q'$ where $f$, $g$, and $q'$, are
  forms of degree $a>0$, $b>0$, and $a+b$, in disjoint sets
  of variables $(x_i)_{i=1,\ldots,m}$, $(y_j)_{j=1,\ldots,n}$,
  and $(z_k)_{k=1,\ldots,p}$, respectively. Let us write $Q$ and $Q'$ for the projective
  hypersurfaces defined by $q$ and $q'$, respectively. Assume that $Q$ is smooth. Under these notations and assumptions, there is a $\bbG_m$-action on $Q$ given by $\lambda\cdot
  (\underline{x},\underline{y},\underline{z}):=(\lambda^b\underline{x},\lambda^{-a}\underline{y},\underline{z})$
  with fixed point locus $\bbP^{m-1}\coprod \bbP^{n-1}\coprod Q'$; note that this implies that $Q'$ is also smooth. By combining Theorem \ref{thm:cellular2} and Example \ref{example:bb} with the fact that $\mathrm{U}_{\mathrm{add}}(\bbP^n)\simeq \mathrm{U}_{\mathrm{add}}(k)^{\oplus (n+1)}$ (see \cite[\S2.4.2]{book}), we hence obtain an induced isomorphism 
\begin{equation}\label{eq:induced-iso}
F(Q) \simeq F(k)^{\oplus (m+n)} \oplus F(Q') 
\end{equation}  
for every additive invariant $F$. Intuitively speaking, isomorphism \eqref{eq:induced-iso} shows that the ``non-trivial parts'' of $F(Q)$ and $F(Q')$ are the same. Finally, recall that the preceding computation holds for all isotropic quadratic forms since it is well-known that these can be written as  $xy+q'(\underline{z})$.
%  The following application of the Bialynicki-Birula decomposition was
%  inspired by \cite{Brosnan}. Let $q=fg+q'$ where $f$, $g$, $q'$ are
%  forms of degree $a>0$, $b>0$ and $a+b$ respectively, in disjoint sets
%  of variables $(x_i)_{i=1,\ldots,m}$, $(y_j)_{j=1,\ldots,n}$,
%  $(z_k)_{k=1,\ldots,p}$.  Let $Q$, $Q'$ be the projective
%  hypersurfaces defined by $q$ and $q'$ and assume $Q$ is smooth.
%  There is a $\bbG_m$-action on $Q$ given by $\lambda\cdot
%  (\underline{x},\underline{y},\underline{z})=(\lambda^b\underline{x},\lambda^{-a}\underline{y},\underline{z})$
%  whose fixed points are $\bbP^{m-1}\coprod \bbP^{n-1}\coprod Q'$ (so
%  in particular $Q'$ is also smooth).
%
%Thanks to 
%Theorem \ref{thm:cellular2} and Example \ref{example:bb}, we  conclude for any additive invariant~$F$:
%\[
%F(Q)\simeq F(k)^{\oplus m+n}\oplus F(Q').
%\]
%So in some sense ``the non-trivial parts'' of $F(Q)$ and $F(Q')$ are the same.
%
%This example applies in particular to isotropic quadratic forms since it is well-known that those can be written as $xy+q'(\underline{z})$.
\end{example}
\section{Motives versus noncommutative motives} 
\label{sec:commnoncomm}
\subsection{Motivic homotopy theory versus noncommutative mixed motives}
%-------------------------------------------------------------------------------
The reduction Theorem \ref{thm:main2} (and Corollary \ref{cor:main2}) allows us to improve the bridge between Morel-Voevodsky's motivic homotopy theory and Kontsevich's noncommutative mixed motives originally constructed in \cite{Bridge}.

Kontsevich introduced in \cite{IAS} a (rigid) symmetric monoidal triangulated category of noncommutative mixed motives $\mathrm{KMM}(k)$. As explained in \cite[\S4]{Bridge}, this category can be described as the smallest thick triangulated subcategory of $\Mot(k)$ containing the objects $\mathrm{U}(\cA)$ with $\cA$ a smooth proper dg category. In the same vein, let us write $\mathrm{KMM}(k)^\oplus$ for the smallest triangulated subcategory of $\Mot(k)$ which contains $\mathrm{KMM}(k)$ and is stable under arbitrary direct sums.

Morel and Voevodsky introduced in \cite{MV,Voevodsky-ICM} the stable $\bbA^1$-homotopy category of $(\bbP^1,\infty)$-spectra $\mathrm{SH}(k)$. By construction, this category comes equipped with a symmetric monoidal functor $\Sigma^\infty(-_+)\colon \mathrm{Sm}(k) \to \mathrm{SH}(k)$ defined on smooth $k$-schemes. Let $\mathrm{KGL} \in \mathrm{SH}(k)$ be the ring $(\bbP^1,\infty)$-spectrum representing homotopy $K$-theory (see \cite{GS,RSO}) and $\mathrm{Mod}(\mathrm{KGL})$ the homotopy category of $\mathrm{KGL}$-modules.
\begin{theorem}\label{thm:new1}
\begin{itemize}
\item[(i)] Let $k$ be a field of characteristic zero. Then, there exists a fully-faithful, symmetric monoidal, triangulated functor $\Phi$ making the following diagram commute
\begin{equation}\label{eq:diagram-1}
\xymatrix{
\mathrm{Sm}(k) \ar[d]_-{\Sigma^\infty(-_+)} \ar[rrr]^-{X \mapsto \perf_\dg(X)} \ar[drr] &&& \dgcat(k) \ar[d]^-{\mathrm{U}} \\
\mathrm{SH}(k) \ar[d]_-{-\wedge \mathrm{KGL}} && \mathrm{KMM}(k) \ar[d]_-{(-)^\vee} \ar[r] & \Mot(k) \ar[d]^-{\uHom(-,\mathrm{U}(k))} \\
\mathrm{Mod}(\mathrm{KGL}) \ar[rr]_-\Phi && \mathrm{KMM}(k)^\oplus \ar[r] & \mathrm{Mot}(k) \,,
}
\end{equation}
where $\uHom(-,-)$ stands for the internal-Hom of the category $\Mot(k)$. %\marginpar{\Michel{I deleted the original (ii) since I assume it follows from (iii).}}
%
%\item[(ii)] When $k$ is a perfect field, there exists a $\bbQ$-linear, fully-faithful, symmetric monoidal, triangulated functor $\Phi_\bbQ$ making the diagram commute:
%\begin{equation}\label{eq:diagram-2}
%\xymatrix{
%\mathrm{Sm}(k) \ar[d]_-{\Sigma^\infty(-_+)_\bbQ} \ar[rrr]^-{X \mapsto \perf_\dg(X)} \ar[drr] &&& \dgcat(k) \ar[d]^-{\mathrm{U}(-)_\bbQ} \\
%\mathrm{SH}(k)_\bbQ \ar[d]_-{-\wedge \mathrm{KGL}_\bbQ} && \mathrm{KMM}(k)_\bbQ \ar[d]_-{(-)^\vee} \ar[r] & \Mot(k)_\bbQ \ar[d]^-{\uHom(-,\mathrm{U}(k)_\bbQ)} \\
%\mathrm{Mod}(\mathrm{KGL}_\bbQ) \ar[rr]_-{\Phi_\bbQ} && \mathrm{KMM}(k)_\bbQ^\oplus \ar[r] & \mathrm{Mot}(k)_\bbQ \,.
%}
%\end{equation}
\item[(ii)] Let $k$ be a perfect field of positive 
    characteristic $p>0$. Then, there exists a
  $\bbZ[1/p]$-linear, fully-faithful, symmetric monoidal,
  triangulated functor $\Phi_{1/p}$ making the following diagram commute (the shorthand $1/p$ stands for $\bbZ[1/p]$):
\begin{equation}\label{eq:diagram-33}
\xymatrix{
\mathrm{Sm}(k) \ar[d]_-{\Sigma^\infty(-_+)_{1/p}} \ar[rrr]^-{X \mapsto \perf_\dg(X)} \ar[drr] &&& \dgcat(k) \ar[d]^-{\mathrm{U}(-)_{1/p}} \\
\mathrm{SH}(k)_{1/p} \ar[d]_-{-\wedge \mathrm{KGL}_{1/p}} && \mathrm{KMM}(k)_{1/p} \ar[d]_-{(-)^\vee} \ar[r] & \Mot(k)_{1/p} \ar[d]^-{\uHom(-,\mathrm{U}(k)_{1/p})} \\
\mathrm{Mod}(\mathrm{KGL}_{1/p}) \ar[rr]_-{\Phi_{1/p}} && \mathrm{KMM}(k)_{1/p}^\oplus \ar[r] & \mathrm{Mot}(k)_{1/p} \,.
}
\end{equation}
%where the shorthand $1/p$ stands for $\bbZ[1/p]$.
\end{itemize}
\end{theorem}
\begin{proof}
%\marginpar{\Michel{Gon\c calo: I rewrote this a bit since the original (ii) no longer exists! Also Theorem \ref{thm:main2} has been rewritten. Can you check that I did not write any stupid things?}}
The outer commutative square of diagram
  \eqref{eq:diagram-1} was constructed in
  \cite[Cor.~2.5(i)]{Bridge}. The inner
  commutative squares follow from Theorem
  \ref{thm:main2}(i) and Corollary \ref{cor:main2} applied to the
  functor $E=\mathrm{U}$. Similarly to the proof of Theorem
  \ref{thm:main2}(ii) (see \S\ref{sec:proofmain2}), one can refine the
  proof of Ayoub \cite[Prop.~2.2.27-2]{Ayoub2} using Gabber's refined
  theory of alterations. Using \cite[Thm.~2.1(iii)]{Bridge},
  we hence obtain the outer commutative square of diagram
  \eqref{eq:diagram-33}. The inner commutative
  squares follow from Theorem
  \ref{thm:main2}(ii) and Corollary \ref{cor:main2} applied to the
  functor $E=\mathrm{U}$.
\end{proof}
Intuitively speaking, Theorem \ref{thm:new1} formalizes the conceptual idea that the diference between the motivic homotopy theory and the theory of noncommutative mixed motives is measured solely by the existence of a $\mathrm{KGL}$-module structure.
\begin{remark}[Morel-Voevodsky's motivic Gysin triangle]\label{rk:motivic1}
Let $X$ be a smooth scheme, $i\colon Z \hookrightarrow X$ a smooth closed subscheme with normal vector bundle $N$, and $j\colon U \hookrightarrow X$ the open complement of $Z$. Making use of 
the Nisnevich topology and of homotopy purity, Morel-Voevodsky constructed in \cite[\S3.2]{MV}\cite[\S4]{Voevodsky-ICM} a motivic Gysin triangle
\begin{equation}\label{eq:Gysin-mot2}
\Sigma^\infty(U_+) \stackrel{\Sigma^\infty(j_+)}{\too} \Sigma^\infty(X_+) \too \Sigma^\infty(\mathrm{Th}(N))\stackrel{\partial}{\too} \Sigma(\Sigma^\infty(U_+))
\end{equation}
in $\mathrm{SH}(k)$, where $\mathrm{Th}(N)$ stands for the Thom space of $N$. Since homotopy $K$-theory is an orientable and periodic cohomology theory, $\Sigma^\infty(\mathrm{Th}(N))\wedge \mathrm{KGL}$ is isomorphic to $\Sigma^\infty(Z_+)\wedge \mathrm{KGL}$. Using the commutative diagram \eqref{eq:diagram-1}, we hence conclude that the image of \eqref{eq:Gysin-mot2} under the composed functor $\Phi \circ (-\wedge \mathrm{KGL})\colon \mathrm{SH}(k) \to \mathrm{KMM}(k)^\oplus$ agrees with the dual of the noncommutative motivic Gysin triangle \eqref{eq:Gysin-mot1}. Roughly speaking, \eqref{eq:Gysin-mot1} is the dual of the $\mathrm{KGL}$-linearization of \eqref{eq:Gysin-mot2}.
\end{remark}
%-------------------------------------------------------------------------------
\subsection{Mixed motives versus noncommutative mixed motives}
%-------------------------------------------------------------------------------
The reduction Theorem \ref{thm:main2} (and Corollary \ref{cor:main2}) allows us also to improve the bridge between Voevodsky's mixed motives and noncommutative mixed motives constructed~in~\cite{Bridge}.

Voevodsky introduced in \cite[\S2]{Voevodsky} the triangulated category of geometric mixed motives $\mathrm{DM}_{\mathrm{gm}}(k)$ (over a perfect field $k$). By construction, this category comes equipped with a symmetric monoidal functor $M\colon \mathrm{Sm}(k) \to \mathrm{DM}_{\mathrm{gm}}(k)$.

Let $\mathrm{HZ} \in \mathrm{SH}(k)$ be the ring $(\bbP^1,\infty)$-spectrum representing motivic cohomology; see \cite[\S6.1]{Voevodsky-ICM}. Thanks to Bloch's work \cite{Bloch}, we have $\mathrm{KGL}_\bbQ \simeq \bigoplus_{i\in \bbZ} \mathrm{HZ}_\bbQ(i)[2i]$. Moreover, $\mathrm{DM}_{\mathrm{gm}}(k)_\bbQ$ identifies with the full triangulated subcategory of compact objects of $\mathrm{DM}(k)_\bbQ:=\Mod(\mathrm{HZ}_\bbQ)$; see \cite{RO1}. As a consequence, base-change along $\mathrm{HZ}_\bbQ \to \mathrm{KGL}_\bbQ$ gives rise to a functor $\mathrm{DM}(k)_\bbQ \to \Mod(\mathrm{KGL}_\bbQ)$. By composing it with $\Phi_\bbQ$, we hence obtain a $\bbQ$-linear, symmetric monoidal, triangulated functor 
\begin{equation}\label{eq:functor-comparison}
\mathrm{R}\colon \mathrm{DM}(k)_\bbQ\too\Mod(\mathrm{KGL}_\bbQ) \stackrel{\Phi_\bbQ}{\too} \mathrm{KMM}(k)_\bbQ^\oplus\,.
\end{equation}
\begin{theorem}\label{thm:new2}
Let $k$ be a perfect field. The functor \eqref{eq:functor-comparison} gives rise to a $\bbQ$-linear, fully-faithful, symmetric monoidal functor $\overline{\mathrm{R}}$ making the following diagram commute
\begin{equation}\label{eq:diagram-3}
\xymatrix{
\mathrm{Sm}(k) \ar[d]_-{M(-)_\bbQ} \ar[rrr]^-{X \mapsto \perf_\dg(X)} \ar[drr] &&& \dgcat(k) \ar[d]^-{\mathrm{U}(-)_\bbQ} \\
\mathrm{DM}_{\mathrm{gm}}(k)_\bbQ \ar[d]_-\pi && \mathrm{KMM}(k)_\bbQ \ar[d]_-{(-)^\vee} \ar[r] & \Mot(k)_\bbQ \ar[d]^-{\uHom(-,\mathrm{U}(k)_\bbQ)} \\
\mathrm{DM}_{\mathrm{gm}}(k)_\bbQ/_{\!\!-\otimes \bbQ(1)[2]} \ar[rr]_-{\overline{\mathrm{R}}} && \mathrm{KMM}(k)_\bbQ \ar[r] & \mathrm{Mot}(k)_\bbQ \,,
}
\end{equation}
where $\mathrm{DM}_{\mathrm{gm}}(k)_\bbQ/_{\!\!-\otimes \bbQ(1)[2]}$ stands for the orbit category of $\mathrm{DM}_{\mathrm{gm}}(k)_\bbQ$ with respect to the Tate motive $\bbQ(1)[2]$ (consult \cite[\S3.5]{Bridge} for the notion of orbit category).
\end{theorem}
\begin{proof}
  The outer commutative square of diagram \eqref{eq:diagram-3} was
  constructed in \cite[Thm.~2.8]{Bridge}. The inner commutative
  squares follow from Theorem \ref{thm:main2} and
  Corollary \ref{cor:main2} applied to the functor $E=\mathrm{U}$.
\end{proof}
Intuitively speaking, Theorem \ref{thm:new2} formalizes the conceptual idea that the commutative world embeds fully-faithfully into the noncommutative world as soon as we ``$\otimes$-trivialize'' the Tate motive $\bbQ(1)[2]$. 
\begin{remark}[Voevodsky's motivic Gysin triangle]\label{rk:motivic2}
Let $X$ be a smooth scheme, $i\colon Z \hookrightarrow X$ a smooth closed subscheme of codimension $c$, and $j\colon U \hookrightarrow X$ the open complement of $Z$. 
Making use of algebraic geometric arguments such as the projective bundle theorem and the deformation to the normal cone, Voevodsky constructed in \cite[\S2]{Voevodsky} a motivic Gysin triangle
\begin{equation}\label{eq:Gysin-mot3}
M(U)_\bbQ \stackrel{M(j)_\bbQ}{\too} M(X)_\bbQ \too M(Z)_\bbQ(c)[2c] \stackrel{\partial}{\too} \Sigma M(U)_\bbQ
\end{equation}
in $\mathrm{DM}_{\mathrm{gm}}(k)_\bbQ$. Using the commutative diagram \eqref{eq:diagram-3}, we hence conclude that the image of \eqref{eq:Gysin-mot3} under the composed functor $\overline{\mathrm{R}} \circ \pi\colon \mathrm{DM}_{\mathrm{gm}}(k)_\bbQ \to \mathrm{KMM}(k)_\bbQ$ agrees with the dual of the rationalized noncommutative motivic Gysin triangle \eqref{eq:Gysin-mot1}. Roughly speaking, \eqref{eq:Gysin-mot1}$_\bbQ$ is the dual of the Tate $\otimes$-trivialization of \eqref{eq:Gysin-mot3}.
\end{remark}
\begin{remark}[Levine's mixed motives]
Levine introduced in \cite[Part I]{Levine} a triangle category of mixed motives $\cD\cM(k)$ and a {\em contravariant} symmetric monoidal functor $h\colon \mathrm{Sm}(k) \to \cD\cM(k)$. As proved in \cite[Part I \S VI Thm.~2.5.5]{Levine}, when $k$ admits resolution of singularities, the assignment $h(X)(n) \mapsto \uHom(M(X),\bbZ(n))$ gives rise to an equivalence of categories $\cD\cM(k) \to \mathrm{DM}_{\mathrm{gm}}(k)$ whose precomposition with $h$ identifies with $X \mapsto M(X)^\vee$. Thanks to Theorem \ref{thm:new2}, there exists then a $\bbQ$-linear, fully-faithful, symmetric monoidal functor $\overline{\mathrm{R}}$ making the diagram commute:
\begin{equation}\label{eq:diagram-4}
\xymatrix{
\mathrm{Sm}(k) \ar[d]_-{h(-)_\bbQ} \ar[rrr]^-{X \mapsto \perf_\dg(X)}&&& \dgcat(k) \ar[dd]^-{\mathrm{U}(-)_\bbQ} \\
\cD\cM(k)_\bbQ \ar[d]_-\pi &&& \\
\cD\cM(k)_\bbQ/_{\!\!-\otimes \bbQ(1)[2]} \ar[rrr]_-{\overline{\mathrm{R}}} &&& \mathrm{KMM}(k)_\bbQ \subset \mathrm{Mot}(k)_\bbQ \,.\quad \quad 
}
\end{equation}
\end{remark}
%-------------------------------------------------------------------------------
\subsection{{\'E}tale descent of noncommutative mixed motives}
%-------------------------------------------------------------------------------
Let $\mathrm{DM}^{\mathrm{et}}(k)$ be the {\'e}tale variant of $\mathrm{DM}(k)$ introduced by Voevodsky in \cite[\S3.3]{Voevodsky}. As proved in {\em loc. cit.}, we have an equivalence of categories $\mathrm{DM}(k)_\bbQ \simeq \mathrm{DM}^{\mathrm{et}}(k)_\bbQ$; consult also Ayoub's ICM survey \cite{Ayoub1}. Theorem \ref{thm:new2} leads then to the following {\'e}tale descent result: 
\begin{theorem}\label{thm:etale}
The presheaf of noncommutative mixed motives
\begin{eqnarray*}
\mathrm{Sm}(k)^\op \too \mathrm{KMM}(k)^\oplus_\bbQ && X \mapsto \mathrm{U}(X)_\bbQ
\end{eqnarray*}
satisfies {\em {\'e}tale descent}, \ie for every $X\in \mathrm{Sm}(k)$ and {\'e}tale cover $\cU=\{U_i \to X\}_{i \in I}$ of $X$, we have an induced isomorphism $\mathrm{U}(X)_\bbQ \simeq \mathrm{holim}_{n\geq 0} \mathrm{U}(\text{\v{C}}_n\cU)_\bbQ$, where $\text{\v{C}}_\bullet\cU$ stands for the \v{C}ech simplicial scheme associated to the cover $\cU$.
\end{theorem}
\begin{proof}
  Thanks to the equivalence of categories $\mathrm{DM}(k)_\bbQ \simeq
  \mathrm{DM}^{\mathrm{et}}(k)_\bbQ$, we have an induced isomorphism
  $M(X)_\bbQ \simeq \mathrm{hocolim}_{n\geq 0}
  M(\text{\v{C}}_n\cU)_\bbQ$ in $\mathrm{DM}(k)_\bbQ$. Since by
  construction the functor \eqref{eq:functor-comparison} preserves
  homotopy colimits, we hence conclude from Theorem \ref{thm:new2}
  that $\mathrm{U}(X)^\vee_\bbQ \simeq
  \mathrm{hocolim}_{n\geq 0}
  \mathrm{U}(\text{\v{C}}_n\cU)^\vee_\bbQ$. The proof
  follows now from the fact that the functor
  $\uHom(-,\mathrm{U}(k)_\bbQ)\colon \Mot(k)_\bbQ \to \Mot(k)_\bbQ$
  interchanges homotopy colimits with homotopy limits and restricts to
  a (contravariant) equivalence of categories $(-)^\vee\colon
  \mathrm{KMM}(k)_\bbQ\to \mathrm{KMM}(k)_\bbQ$.
\end{proof}
%\noindent\textbf{Notations.} Throughout the article, $k$ will denote a base field. Unless stated differently, all tensor products will be taken over $k$.
\medbreak\noindent\textbf{Acknowledgments.} G.~Tabuada is very grateful to Joseph Ayoub for useful ``motivic'' discussions. He also would like to thank the hospitality of the Department of Mathematics of the University of Hasselt, Belgium, where this work was initiated. The authors are also grateful to Mikhail Bondarko for the references \cite{Bondarko,Kelly}.
%-------------------------------------------------------------------------------
\section{Preliminaries}
%-------------------------------------------------------------------------------
%-------------------------------------------------------------------------------
\subsection{Dg categories}\label{sub:dg}
%-------------------------------------------------------------------------------
Let $(\cC(k),\otimes, k)$ be the category of (cochain) complexes of %\marginpar{\Michel{Gon\c calo: I added some extra sections here with things that before were scattered over the body of the manuscript.}}
$k$-vector spaces; we use cohomological notation. A {\em differential
  graded (=dg) category $\cA$} is a category enriched over $\cC(k)$
and a {\em dg functor} $F:\cA\to \cB$ is a functor enriched over
$\cC(k)$; for further details consult Keller's ICM survey
\cite{ICM-Keller}.
%-------------------------------------------------------------------------------
%\subsection{Dg Modules}
%-------------------------------------------------------------------------------

Let $\cA$ be a dg category. The opposite dg category $\cA^\op$ has the
same objects as $\cA$ and $\cA^\op(x,y):=\cA(y,x)$. A {\em right dg
  $\cA$-module} is a dg functor $\cA^\op \to \cC_\dg(k)$ with values
in the dg category $\cC_\dg(k)$ of complexes of $k$-vector spaces. Let
us write $\cC(\cA)$ for the category of right dg
$\cA$-modules. Following \cite[\S3.2]{ICM-Keller}, the derived
category $\cD(\cA)$ of $\cA$ is defined as the localization of
$\cC(\cA)$ with respect to the objectwise quasi-isomorphisms. Let
$\cD_c(\cA)$ be the triangulated subcategory of compact objects.
%-------------------------------------------------------------------------------
%\subsection{Morita equivalences}\label{sub:Morita}
%-------------------------------------------------------------------------------

A dg functor $F:\cA\to \cB$ is called a {\em Morita equivalence} if it
induces an equivalence of categories $\cD(\cB) \stackrel{\simeq}{\to}
\cD(\cA)$; see \cite[\S4.6]{ICM-Keller}. As proved in
\cite[Thm.~5.3]{Additive}, $\dgcat(k)$ admits a Quillen model
structure whose weak equivalences are the Morita equivalences. Let us
denote by $\Hmo(k)$ the associated homotopy category.
%-------------------------------------------------------------------------------
%\subsection{Tensor product}
%-------------------------------------------------------------------------------

The {\em tensor product $\cA\otimes\cB$} of dg categories is defined
as follows: the set of objects is the cartesian product and
$(\cA\otimes\cB)((x,w),(y,z)):= \cA(x,y) \otimes \cB(w,z)$. As
explained in \cite[\S2.3]{ICM-Keller}, this construction gives rise to
a symmetric monoidal structure on $\dgcat(k)$, which descends to the
homotopy category
$\Hmo(k)$. %After deriving it $-\otimes-$, this symmetric monoidal structure descends to $\Hmo(k)$; consult \cite[\S4.3]{ICM-Keller} for details.
%-------------------------------------------------------------------------------
%\subsection{Bimodules}
%-------------------------------------------------------------------------------

An {\em $\cA\text{-}\cB$-bimodule $\mathrm{B}$} is a dg functor
$\mathrm{B}\colon \cA\otimes \cB^\op \to \cC_\dg(k)$ or equivalently a
right dg $(\cA^\op \otimes \cB)$-module. A standard example is the
$\cA\text{-}\cB$-bimodule
\begin{eqnarray}\label{eq:bimodule2}
{}_F\mathrm{B}:\cA\otimes \cB^\op \to \cC_\dg(k) && (x,z) \mapsto \cB(z,F(x))
\end{eqnarray}
associated to a dg functor $F:\cA\to \cB$.
%-------------------------------------------------------------------------------
%\subsection{Smooth proper dg categories}\label{sub:smooth}
%-------------------------------------------------------------------------------

Recall from Kontsevich \cite{IAS,ENS,Miami,finMot} that a dg category $\cA$ is called {\em smooth} if the $\cA\text{-}\cA$-bimodule ${}_{\id}\cB$ belongs to the triangulated category $\cD_c(\cA^\op\otimes \cA)$ and {\em proper} if $\sum_i \mathrm{dim}\, H^i\cA(x,y)< \infty$ for any ordered pair of objects $(x,y)$. Examples include the finite dimensional $k$-algebras of finite global dimension (when $k$ is perfect) as well as the dg categories $\perf_\dg(Y)$ associated to smooth proper $k$-schemes $Y$.
%-------------------------------------------------------------------------------
\subsection{Localizing invariants}\label{sub:localizing}
%-------------------------------------------------------------------------------
Let $E\colon \dgcat(k) \to \cT$ be a functor, with values in a
triangulated category, which inverts Morita equivalences. Thanks to
the universal property of the homotopy category $\Hmo(k)$, we have an
induced functor $E\colon \Hmo(k) \to \cT$. Recall from
\cite[Thm.~4.11]{ICM-Keller} that the homotopy category $\Hmo(k)$ is
pointed\footnote{The dg category with one object and one morphism is
  the initial=terminal object of $\Hmo(k)$.} and that a short exact
sequence of dg categories
\begin{equation}\label{eq:ses}
0 \too \cA \stackrel{I}{\too} \cB \stackrel{P}{\too} \cC \too 0
\end{equation}
consists of morphisms $I$ and $P$ in $\Hmo(k)$ such that $P\circ I=0$, $I$ is the kernel of $P$, and $P$ is the cokernel of $I$. A ``generic'' example is given by the Drinfeld's dg quotient $\cA \subset \cB \to \cB/\cA$ of an inclusion of dg categories; consult \cite{Drinfeld} for details.
\begin{definition}\label{def:localizing}
A functor $E\colon \dgcat(k) \to \cT$ as above is called a {\em localizing invariant} if the induced functor $E\colon \Hmo(k) \to \cT$ sends   short exact sequences of dg categories \eqref{eq:ses} to distinguished triangles 
$$ E(\cA) \stackrel{E(I)}{\too} E(\cB) \stackrel{E(P)}{\too} E(\cC) \stackrel{\partial}{\too} \Sigma E(\cA)\,.$$
in a way which is functorial for strict morphisms of exact sequences.
\begin{remark} Using the methods in \cite{Exact}, one may show
    that the functoriality of $E$ on strict morphisms of exact
    sequences of dg categories implies that $E$ is functorial on
    morphisms between exact sequences of dg categories in $\Hmo(k)$.
 % \marginpar{\Michel{Gon\c calo: should we say more here?}}
\end{remark}
%the short exact sequences of dg categories \eqref{eq:ses} to distinguished triangles
%\begin{eqnarray*}
%\cA\stackrel{I}{\too} \cB \stackrel{P}{\too} \cC & \mapsto & E(\cA) \stackrel{E(I)}{\too} E(\cB) \stackrel{E(P)}{\too} E(\cC) \stackrel{\partial}{\too} \Sigma E(\cA)
%\end{eqnarray*} 
%and morphisms of short exact sequences of dg categories to morphisms of distinguished triangles: 
%\begin{eqnarray}\label{eq:assignment}
%\xymatrix{
%A \ar[d] \ar[r] & \cB \ar[d] \ar[r] & \cC \ar[d] & E(\cA) \ar[d] \ar[r] & E(\cB) \ar[d] \ar[r] & E(\cC) \ar[d] \ar[r]^-\partial & \Sigma E(\cA) \ar[d] \\
%\cA' \ar[r] & \cB' \ar[r] & \cC' & E(\cA') \ar[r] & E(\cB') \ar[r] & E(\cC') \ar[r]_-\partial & \Sigma E(\cA')\,.}
%\end{eqnarray}
%Moreover, the assignment \eqref{eq:assignment} is compatible with the composition of morphisms.
\end{definition}
\begin{example}[Mixed complex]
Following Kassel \cite{Kassel}, a {\em mixed complex} is a (right) dg module over the algebra of dual numbers $\Lambda:=k[\epsilon]/\epsilon^2$ with $\mathrm{deg}(\epsilon)=-1$ and $d(\epsilon)=0$. As proved by Keller in \cite[\S1.5]{Exact}, the mixed complex construction gives rise to a localizing invariant $C\colon \dgcat(k) \to \cD(\Lambda)$. Since Hochschild homology, cyclic homology, negative cyclic homology, and periodic cyclic homology factor through $C$, they are also examples of localizing invariants; consult \cite[\S2.2]{Kel1} for details.
\end{example}
\begin{remark}[Quillen model]\label{rk:generic}
Let $\cM$ be a stable Quillen model category and $E\colon \dgcat(k) \to \cM$ a functor which sends the Morita equivalences to weak equivalences and the Drinfeld's dg quotients $\cA \subset \cB \to \cB/\cA$ to homotopy (co)fiber sequences $E(\cA) \to E(\cB) \to E(\cB/\cA)$. Consider the associated composition
\begin{equation}\label{eq:associated}
\dgcat(k) \stackrel{E}{\too} \cM \too \Ho(\cM)
\end{equation}
with values in the homotopy triangulated category. Clearly, the functor \eqref{eq:associated} inverts Morita equivalences. Moreover, it sends in a functorial way the Drinfeld's dg quotients to distinguished triangles. As proved by Keller in \cite[\S4]{Exact}, every short exact sequence of dg categories \eqref{eq:ses} can be ``strictified'', in a functorial way, into the Drinfeld's dg quotient $\cA \subset \cB \to \cB/\cA$ of an inclusion of dg categories. This hence implies that \eqref{eq:associated} is a localizing invariant.
\end{remark}
\begin{example}
Examples \ref{ex:KH}-\ref{ex:etale} and \ref{ex:Mot} fit into the framework of Remark \ref{rk:generic}. Further examples include nonconnective algebraic $K$-theory, topological Hochschild homology, topological cyclic homology, etc; consult \cite{book} for details.
\end{example}
\subsection{Universal additive invariant and its relation with NC motives}
\label{sec:universal}
  We start by recalling from \cite{Additive} the construction of the
  universal additive invariant. %Let $\cA$ and $\cB$ be two dg categories. 
  As proved in \cite[Cor.~5.10]{Additive}, there is a
  natural bijection between $\Hom_{\Hmo(k)}(\cA,\cB)$ and the set of
  isomorphism classes of the full triangulated subcategory
  $\rep(\cA,\cB)\subset \cD(\cA^\op \otimes \cB)$ of those
  $\cA\text{-}\cB$-bimodules $\mathrm{B}$ such that for every $x \in
  \cA$ the right dg $\cB$-module $\mathrm{B}(x,-)$ belongs to
  $\cD_c(\cB)$. Under this bijection, the composition law of $\Hmo(k)$
  corresponds to the tensor product of bimodules. Since the bimodules
  \eqref{eq:bimodule2} belong to $\rep(\cA,\cB)$, we have the tautological functor
\begin{eqnarray}\label{eq:functor1}
\dgcat(k)\to \Hmo(k) & \cA \mapsto \cA & F \mapsto {}_F \mathrm{B}\,.
\end{eqnarray}
The {\em additivization} of $\Hmo(k)$ is the additive category
$\Hmo_0(k)$ with the same objects and with abelian groups of morphisms
$\Hom_{\Hom_0(k)}(\cA,\cB)$ given by the Grothendieck group
$K_0\rep(\cA,\cB)$ of the triangulated category
$\rep(\cA,\cB)$. By construction, we have the following functor
 \begin{eqnarray}\label{eq:functor2}
\Hmo(k)\to \Hmo_0(k) & \cA \mapsto \cA & \mathrm{B} \mapsto [\mathrm{B}]\,.
\end{eqnarray}
Let us denote by $\mathrm{U}_{\mathrm{add}}$ the composition
\eqref{eq:functor2}$\circ$\eqref{eq:functor1}. As proved in
\cite[Thms.~5.3 and 6.3]{Additive}, the functor
$\mathrm{U}_{\mathrm{add}}\colon \dgcat(k) \to \Hmo_0(k)$ is the {\em
  universal additive invariant}, \ie given any additive category
$\mathrm{A}$ we have an induced equivalence of categories
\begin{equation}\label{eq:equivalence}
  \mathrm{U}_{\mathrm{add}}\colon \Fun_{\mathrm{additive}}(\Hmo_0(k),\mathrm{A}) \stackrel{\simeq}{\too} \Fun_{\mathrm{add}}(\dgcat(k),\mathrm{A})\,,
\end{equation}
where the left-hand side denotes the category of additive functors and
the right-hand side the category of additive invariants (see \S\ref{sec:additive}). 

As mentioned
in \S\ref{sec:additive}, every localizing invariant is in particular
an additive invariant. Therefore, since the functor $\mathrm{U}:\dgcat(k)\r \Mot(k)$ (see Example \ref{ex:Mot}) is a localizing invariant, and hence an additive invariant, it factors as follows:
\begin{equation}
\label{eq:locadd}
\xymatrix{
\dgcat(k) \ar[d]_-{\mathrm{U}_{\mathrm{add}}} \ar[rr]^{\mathrm{U}} &&  \Mot(k) &\\
\Hmo_0(k) \ar@/_0.5pc/[urr]_{\overline{\mathrm{U}}} && \,.
}
\end{equation}
The following Proposition \ref{prop:K-theory1} and Lemma \ref{lem:invertible}, concerning the functors $\mathrm{U}$ and $\mathrm{U}_{\mathrm{add}}$, will play a key role in the proof of Theorem \ref{thm:cellular2}.
%We collect some results concerning the universal invariants $\mathrm{U}$, $\mathrm{U}_{\add}$ which will be used below.
\begin{proposition}\label{prop:K-theory1}
Given a smooth proper $k$-scheme $X$ and a smooth $k$-scheme $Y$, there are isomorphisms of abelian groups
\begin{eqnarray}\label{eq:K-isos1}
\Hom_{\Mot(k)}(\mathrm{U}(X),\Sigma^n\mathrm{U}(Y))\simeq K_{-n}(X\times Y) && n \in \bbZ\,.
\end{eqnarray}
In particular, the abelian group \eqref{eq:K-isos1} is zero whenever $n>0$.
\end{proposition}
\begin{proof}
As proved in \cite[Cor.~2.7]{A1-homotopy}, the left-hand side of \eqref{eq:K-isos1} is isomorphic to $KH_{-n}(X \times Y)$. The proof follows then from the fact that homotopy $K$-theory agrees with Quillen's algebraic $K$-theory on smooth schemes.
\end{proof}
\begin{lemma}\label{lem:invertible}
Given smooth $k$-schemes $X$ and $Y$, with $X$ proper, the morphism
\begin{equation}\label{eq:induced-last}
\Hom_{\Hmo_0(k)}(\mathrm{U}_{\mathrm{add}}(X),\mathrm{U}_{\mathrm{add}}(Y))\too \Hom_{\Mot(k)}(\mathrm{U}(X),\mathrm{U}(Y))\,,
\end{equation}
induced by the additive functor $\overline{\mathrm{U}}$, is invertible.
\end{lemma}
\begin{proof}
  Thanks to Proposition \ref{prop:K-theory1}, the right-hand side of
  \eqref{eq:induced-last} is naturally isomorphic to $K_0(X\times
  Y)$. For the left-hand side we have the
  identifications
\begin{eqnarray}
  \Hom_{\Hmo_0(k)}(\mathrm{U}_{\mathrm{add}}(X),\mathrm{U}_{\mathrm{add}}(Y))
&\simeq &  K_0(\rep(\perf_\dg(X)^\op,\perf_\dg(Y))) \label{eq:mo1}\nonumber\\
 & \simeq & K_0(\perf_\dg(X)^\op \otimes \perf_\dg(Y)) 
\label{eq:mo2} \\
  & \simeq & K_0(\perf_\dg(X)\otimes \perf_\dg(Y))\label{eq:mo3}\\
  & \simeq & K_0(\perf_{\dg}(X\times Y))\label{eq:mo4}\\
  & \simeq & K_0(X\times Y) \nonumber
\,, \label{eq:2}
\end{eqnarray}
where \eqref{eq:mo2} follows from the fact that $\rep(\cA,\cB)\simeq
\cD_c(\cA^\op \otimes \cB)$ for any two dg categories with $\cA$ smooth and proper,
\eqref{eq:mo3} follows from the Morita equivalence
$\perf_\dg(X)^\op \to \perf_\dg(X)$ given by $\cF \mapsto
\underline{\Hom}_X(\cF,\cO_X)$, and \eqref{eq:mo4} from
Lemma \ref{lem:product}. The proof follows now from the construction of the diagram \eqref{eq:locadd}.
\end{proof}

\subsection{Quasi-coherent sheaves and their derived categories}
Given a quasi-compact quasi-separated $k$-scheme $X$, let us write
$\mathrm{Mod}(X)$ for the Grothendieck category of $\cO_X$-modules,
$\mathrm{Qcoh}(X)$ for the full subcategory of quasi-coherent
$\cO_X$-modules, $\cD(X):=\cD(\mathrm{Mod}(X))$ for the derived
category of $X$, and $\cD_{\mathrm{Qcoh}}(X)\subset \cD(X)$ for the
full triangulated subcategory of those complexes of $\cO_X$-modules
with quasi-coherent cohomology. In the same vein, given a closed
subscheme $Z \hookrightarrow X$, let us write $\cD(X)_Z\subset \cD(X)$
and $\cD_{\mathrm{Qcoh}}(X)_Z \subset \cD_{\mathrm{Qcoh}}(X)$ for the
full triangulated subcategories of those complexes of $\cO_X$-modules
that are supported on $Z$.
\begin{theorem}[Compact generation]
\label{th:cg}
Assume that the open complement of $Z$ is also quasi-compact quasi-separated. Under this assumption, the triangulated category $\cD_{\mathrm{Qcoh}}(X)_Z$ is compactly generated. Moreover, its full triangulated subcategory of compact objects identifies with $\perf(X)_Z$.
\end{theorem}
\begin{proof}
Simply imitate the proof of \cite[Thm.~3.1.1]{BV}; consult also \cite{Neeman3}\cite[Tag 0AEC, Lem. 62.14.5]{Stacks}.
\end{proof}
The following result will play a key role in the proof of Theorem \ref{thm:Nisnevich}.
\begin{proposition}\label{prop:SES} 
  Let $X$ be a quasi-compact quasi-separated scheme, $p\colon V \hookrightarrow X$ a
  quasi-compact open subscheme, and $W \hookrightarrow X$ the closed
  complement of $V$. For every closed subscheme $i\colon Z
  \hookrightarrow X$ with quasi-compact complement, we have an induced short exact sequence of dg
  categories
\begin{equation}\label{eq:SES-support}
0 \too \perf_\dg(X)_{Z\cap W} \too \perf_\dg(X)_Z \stackrel{p^\ast}{\too} \perf_\dg(V)_{Z\cap V}\too 0\,.
\end{equation}
\end{proposition}
\begin{proof}
  As explained by Keller in \cite[Thm.~4.11]{ICM-Keller},
  \eqref{eq:SES-support} is a short exact sequence of dg categories if
  and only if the associated sequence of triangulated categories
\begin{equation}
\label{eq:verdier}
\perf(X)_{Z\cap W}\too \perf(X)_Z
  \stackrel{p^\ast}{\too} \perf(V)_{Z\cap V}
  \end{equation}
is exact in the sense
  of Verdier. We claim that there is a short exact sequence:
\def\Qch{\operatorname{Qcoh}} \def\cone{\operatorname{cone}}
\begin{equation}
\label{eq:verdier2}
0 \too \cD_{\Qch}(X)_{Z\cap W}\too \cD_{\Qch}(X)_Z
  \stackrel{p^\ast}{\too} \cD_{\Qch}(V)_{Z\cap V}\too0 \,.
\end{equation}
%\marginpar{\Goncalo{Michel: I am not sure that we should write $Rp_\ast$ because throughout the article we are (almost) never writing the left/right derived functors...}}
%\marginpar{\Michel{Gon\c calo: If you prefer to write just $p_\ast$ then that's ok with me. Note that I deleted the equality signs since those were not what I meant. I meant to say that to prove that $Rp_\ast\cF\in D_{\operatorname{Qcoh}(X)_Z}$ show that the restriction of $Rp_\ast \cF$ to the complement of $Z$ is zero. To do the latter one used compatibility of $Rp_\ast$ with base change.}}
As usual, this follows from the following facts: (i) if $\cF \in \cD_{\Qch}(V)_{Z\cap V}$, then $p_\ast(\cF)$ belongs to $\cD_{\Qch}(X)_Z$; (ii) if $\cG\in \cD_{\Qch}(X)_Z$, then $\cone(\cG\rightarrow p_\ast
  p^\ast(\cG))$ belongs to $\cD_{\Qch}(X)_{Z\cap W}$. In what concerns (i), note that by base change the restriction of $p_\ast(\cF)$ to the complement of $Z$ is zero. In what concerns (ii), restrict to $V$. Thanks to Theorem \ref{th:cg}, the   category $\cD_{\Qch}(X)_{Z\cap W}$ is generated by perfect
  complexes. Therefore, by applying Neeman's celebrated result
  \cite[Thm.~2.1]{Neeman2} to \eqref{eq:verdier2}, we conclude that
  \eqref{eq:verdier} is also a short exact sequence of triangulated
  categories.
%This follows in the usual way from the fact that if $\cF\in
%\cD_{\Qch}(V)_{Z\cap V}$ then $Rp_\ast\cF\in \cD_{\Qch}(X)_Z$ (restrict to the complement of $Z$) and moreover
%if $\cG\in \cD_{\Qch}(X)_Z$ then $\cone(\cG\rightarrow Rp_\ast
%p^\ast\cG)\in \cD_{\Qch}(X)_{Z\cap W}$ (restrict to $V$). 
%By Theorem \ref{th:cg}  $\cD_{\Qch}(X)_{Z\cap W}$ is generated by perfect complexes on $X$. 
%Then \eqref{eq:verdier} follows from \eqref{eq:verdier2} using \cite[Thm 2.1]{Neeman2}.
\end{proof}
%Recall also the following result.
The following ``excision'' result will be used in the proof of Theorems~\ref{thm:Nisnevich}~and~\ref{thm:formality}.
\begin{theorem}[{see \cite[Thm 2.6.3]{TT}}]
\label{th:tt}
Let $f:X'\r X$ be a flat morphism of quasi-compact quasi-separated $k$-schemes and let $Z\hookrightarrow X$ be a closed subscheme with
quasi-compact complement such that $Z':=X'\times_X Z\r Z$ is an isomorphism of $k$-schemes. Then, the functors $(f_\ast, f^\ast)$ define inverse equivalences of categories between $\cD(X)_Z$ and $\cD(X')_{Z'}$ and between $\perf(X)_Z$ and $\perf(X')_{Z'}$.
\end{theorem}
\begin{proof} 
As proved in {\em loc. cit.}, the functors $(f_\ast, f^\ast)$ define inverse equivalences between $\cD^-(X)_Z$ and $\cD^-(X')_{Z'}$. However, since $\cD(X)_Z$ and $\cD(X')_{Z'}$ admit arbitrary direct sums and are generated by $\cD^-(X)_Z$ and $\cD^-(X')_{Z'}$, respectively, we conclude that the functors $(f_\ast, f^\ast)$ also define inverse equivalences between $\cD(X)_Z$ and $\cD(X')_{Z'}$. By restriction to compact objects, we hence obtain inverse equivalences between $\perf(X)_Z$ and $\perf(X')_{Z'}$; see Theorem \ref{th:cg}.
%In loc.\ cit.\ it is stated that $(Rf_\ast, f^\ast)$
%  define inverse equivalences between $\cD^-(X)_Z$ and $\cD^-(X')_{Z'}$. However since $\cD(X)_Z$ and $\cD(X')_{Z'}$ admit arbitrary coproducts and are generated by $\cD^-(X)_Z$ and $\cD^-(X')_{Z'}$ we also obtain an equivalence between
%$\cD(X)_Z$ and $\cD(X')_{Z'}$. The claim about perfect complexes follows from the fact that the perfect complexes coincide with the compact objects. See Theorem \ref{th:cg}.
\end{proof}
%Finally we recall the following basic result.
The following result will be used in the proof of Generalization (G2).
\begin{lemma} \label{lem:product} Let $X$ and $Y$ be two quasi-compact quasi-separated $k$-schemes. Then, there is a Morita equivalence:
\begin{eqnarray*}
\perf_{\dg}(X)\otimes \perf_{\dg}(Y)\r \perf_{\dg}(X\times Y) && (\cF, \cF')\mapsto \cF\boxtimes \cF'\,.
\end{eqnarray*}
\end{lemma} \def\Qcoh{\mathrm{Qcoh}}
\begin{proof}
Let $\cG$ and $\cG'$ be compact generators of the triangulated categories $\cD_{\Qcoh}(X)$ and $\cD_{\Qcoh}(Y)$, respectively; see Theorem \ref{th:cg}. According to \cite[Lem.~3.4.1]{BV}, $\cG\boxtimes \cG'$ is a compact generator of $\cD_{\Qcoh}(X\times Y)$. Hence, it is sufficient to show that the dg $k$-algebras $\REnd_{X\times Y}(\cG\boxtimes \cG')$ and $\REnd_X(\cG)\otimes\REnd_Y(\cG')$ are quasi-isomorphic. Let $p\colon X \times Y \to X$ and $q\colon X\times Y \to Y$ be the projection maps. Since $p^\ast(\cG)$ and $q^\ast(\cG')$ are perfect complexes, we have
$$ \REnd_{X\times Y}(\cG\boxtimes \cG')={\bf R}\Gamma(X\times Y, {\bf R}\underline{\mathrm{End}}_{X\times Y}(p^\ast(\cG)\otimes^{{\bf L}}_{X\times Y}  q^\ast(\cG')))$$
%and
\begin{eqnarray*}
\bold{R}\underline{\mathrm{End}}_{X\times Y}(p^\ast(\cG)\otimes^{\bf L}_{X\times Y}  q^\ast(\cG')) & \simeq & p^\ast(\bold{R}\underline{\mathrm{End}}_X(\cG)) \otimes^{\bf L}_{X\times Y} q^\ast(\bold{R}\underline{\mathrm{End}}_Y(G')) \\
& \simeq & \bold{R}\underline{\mathrm{End}}_X(\cG)\boxtimes \bold{R}\underline{\mathrm{End}}_Y(\cG')\,.
\end{eqnarray*}
Therefore, it suffices to show that $\bold{R}\Gamma(X\times Y,\cF\boxtimes \cF')\simeq\bold{R}\Gamma(X,\cF)\otimes \bold{R}\Gamma(Y,\cF')$ for any two complexes with quasi-coherent cohomology $\cF$ and $\cF'$. We have
\begin{eqnarray}
\bold{R}\Gamma(X\times Y,\cF\boxtimes \cF')&\simeq & \bold{R}\Gamma(X\times Y,p^\ast(\cF)\otimes^{\bf L}_{X\times Y} q^\ast(\cF')) \nonumber\\
&\simeq &\bold{R}\Gamma(Y,q_\ast p^\ast(\cF)\otimes^{\bf L}_Y(\cF')) \label{eq:number-1}\\
&\simeq & \bold{R}\Gamma(Y,(\bold{R}\Gamma(X,\cF)\otimes_k \cO_Y)\otimes^{\bf L}_Y \cF') \label{eq:number-2}\\
&\simeq & \bold{R}\Gamma(Y,\bold{R}\Gamma(X,\cF)\otimes_k\cF') \nonumber\\
&\simeq & \bold{R}\Gamma(X,\cF)\otimes \bold{R}\Gamma(Y,\cF')\,, \nonumber
\end{eqnarray}
where \eqref{eq:number-1} follows from the projection formula for $q$ and \eqref{eq:number-2} from flat base change for $\bold{R}\Gamma(Y,-)$. This concludes the proof.
%Let $G$, $G'$ be respectively compact generators for $\cD_{\Qcoh}(X)$ and $\cD_{\Qcoh}(Y)$. According to
%\cite{BV} $G\boxtimes G'$ is a compact generator for $\cD_{\Qcoh}(X\times Y)$. Hence  it is sufficient to show that 
%$\REnd_{X\times Y}(G\boxtimes G')=\REnd_X(G)\otimes_k \REnd_Y(G')$. Letting $p,q$ be the projections on $X$ and $Y$ we have
%$
%\REnd_{X\times Y}(G\boxtimes G')=\bold{R}\Gamma(X\times Y, \bold{R}\cE\mathit{nd}_{X\times Y}(p^\ast G\Lotimes_{X\times Y}  q^\ast G'))
%$ and $\bold{R}\cE\mathit{nd}(p^\ast G\Lotimes_{X\times Y}  q^\ast G')= p^\ast\bold{R}\cE\mathit{nd}_X(G)
%\Lotimes_{X\times Y} q^\ast \bold{R}\cE\mathit{nd}_Y(G')=\bold{R}\cE\mathit{nd}_X(G)\boxtimes \bold{R}\cE\mathit{nd}_Y(G')$ since $p^\ast G$ and $q^\ast G'$ are perfect.
%Hence it is sufficient to show that if $\cA$, $\cA'$ are  complexes with quasi-coherent cohomology on $X$ and $Y$ then
%$\bold{R}\Gamma(X\times Y,\cA\boxtimes \cA')=\bold{R}\Gamma(X,\cA)\otimes \bold{R}\Gamma(Y,\cA')$. We compute
%\begin{align*}
%\bold{R}\Gamma(X\times Y,\cA\boxtimes \cA')&=\bold{R}\Gamma(X\times Y,p^\ast \cA\Lotimes_{X\times Y} q^\ast\cA')\\
%&=\bold{R}\Gamma(Y,Rq_\ast p^\ast \cA\Lotimes_Y \cA')\\
%&=\bold{R}\Gamma(Y,(\bold{R}\Gamma(X,\cA)\otimes_k \cO_Y)\Lotimes_Y \cA')\\
%&=\bold{R}\Gamma(Y,\bold{R}\Gamma(X,\cA)\otimes_k\cA')\\
%&=\bold{R}\Gamma(X,\cA)\otimes_k \bold{R}\Gamma(Y,\cA')
%\end{align*}
%The second equality is the projection formula for $q$. The third equality is flat base change for $\bold{R}\Gamma(Y,-)$.
\end{proof}
%-------------------------------------------------------------------------------
\subsection{Notation}\label{sub:notations}
%-------------------------------------------------------------------------------
Let $X$ be a $k$-scheme, $Z\hookrightarrow X$ a closed
subscheme, $\cA$ a dg category, and
$E\colon\dgcat(k)\to \cC$ a functor with values in an arbitrary category. In order to simplify the exposition, we will write \[
E(X;Z;\cA):=E(\perf(X)_Z\otimes \cA)\,.
\]
If $Z=X$ or $\cA=k$, then we will omit the corresponding symbols from the notation.
%-------------------------------------------------------------------------------
\section{Nisnevich descent in the supported setting}\label{sec:Nisnevich}
%-------------------------------------------------------------------------------
Consider the cartesian square of quasi-compact quasi-separated $k$-schemes
\begin{equation}\label{eq:Nisnevich}
\xymatrix{
V_{12}:=V_1 \times_X V_2 \ar[d] \ar[r] & V_2 \ar[d]^-{p_2} \\
V_1 \ar[r]_-{p_1} & X\,,
}
\end{equation}
where $p_1$ is an open immersion and $p_2$ is an {\'e}tale map
inducing an isomorphism of reduced schemes $p_2^{-1}(X\backslash
V_1)_{\mathrm{red}}\simeq (X\backslash V_1)_{\mathrm{red}}$. As proved
by Morel and Voevodsky in \cite[\S3.1 Prop. 1.4]{MV}, the
\index{Nisnevich topology}Nisnevich topology is generated by the
distinguished squares \eqref{eq:Nisnevich}. The Zariski topology is
generated by those distinguished squares \eqref{eq:Nisnevich} in which
$p_2$ is also an open immersion.

\begin{notation} A sequence of maps $a \to a' \to a'' \to \Sigma a$ in a
    triangulated category $\cT$ is called a \emph{LES-triangle} if it
    becomes a Long Exact Sequence after applying $\Hom_{\cT}(b,-)$,
    for every object $b$ of $\cT$. Distinguished triangles are, of
    course, LES-triangles but the converse is false.
%Let us say that a sequence of maps
%$
%A\r B\r C\r \Sigma A
%$
%in a triangulated category $\cT$ is a \emph{weak triangle} if it becomes a long exact sequence after \marginpar{\Michel{Gon\c calo: I am not really happy with the terminology ``weak triangle''. Neeman has a notion of pre-triangle but unfortunately that is considerably stronger. I recall having seen a paper once discussing all kinds of triangles, but I can't find it now. Do you have an alternative suggestion?}}
%applying $\Hom_{\cT}(b,-)$ for any $b\in \Ob(\cT)$. Distinguished triangles are of course
%weak triangles but the converse is false.
\end{notation}

Let $i\colon Z\hookrightarrow X$ be a closed subscheme with
quasi-compact complement, $Z_1:=Z\cap V_1$, $Z_2:=p_2^{-1}(Z)$, and
$Z_{12}:=Z_1 \times_Z Z_2$. Under these notations, and those of
\S\ref{sub:notations}, we have the following Nisnevich descent result:
\begin{theorem}\label{thm:Nisnevich}
For every localizing invariant $E\colon \dgcat(k) \to \cT$ we have a ``Mayer-Vietoris'' LES-triangle\footnote{If the functor $E$ is suitably enhanced, then \eqref{eq:Nisnevich-1} can be made into an actual distinguished triangle. However, we will
not need this extra assumption/result.}
\begin{equation}
\label{eq:Nisnevich-1}
E(X;Z) \stackrel{\pm}{\to} E(V_1;Z_1)\oplus E(V_2;Z_2) \to E(V_{12};Z_{12}) \stackrel{\delta}{\to} \Sigma E(X;Z)\,.
\end{equation}
%in $\cT$.
\end{theorem}
%\begin{remark}\label{rk:nosupport1}
%When $Z=\varnothing$, \eqref{eq:MV} reduces to the Mayer-Vietoris long exact sequence:
%\begin{equation}\label{eq:MV1}
%\cdots \to E^b_{n}(X) \stackrel{\pm}{\to} E^b_{n}(V_1)\oplus E^b_{n}(V_2) \to E^b_{n}(V_{12}) \stackrel{\delta_n}{\to} E^b_{n-1}(X) \to \cdots
%\end{equation}
%\end{remark}
%\begin{remark} 
%\end{remark}
\begin{proof}
Let us write $W$ for the (reduced) closed complement $(X\backslash V_1)_{\mathrm{red}}$ of $V_1$. Under this notation, we have the following commutative diagram
\begin{equation}\label{eq:perfect1}
\xymatrix{
\perf_\dg(X)_{Z\cap W} \ar[d] \ar[r] & \perf_\dg(X)_Z \ar[d]_-{p_2^\ast} \ar[r]^-{p_1^\ast} & \perf_\dg(V_1)_{Z_1} \ar[d] \\
\perf_\dg(V_2)_{Z_2\cap p_2^{-1}(W)} \ar[r] & \perf_\dg(V_2)_{Z_2} \ar[r] & \perf_\dg(V_{12})_{Z_{12}}
}
\end{equation}
in the homotopy category $\Hmo(k)$. Thanks to Proposition
\ref{prop:SES}, both rows are short exact sequence of dg
categories. Moreover, since $p_2^{-1}(Z\cap W)=Z_2\cap p_2^{-1}(W)$ and $p_2\colon V_2\r X$ is \'etale, the left-hand side
vertical morphism is a Morita equivalence by Theorem \ref{th:tt}. By applying $E$ to
\eqref{eq:perfect1}, we  obtain an induced morphism between
distinguished triangles with invertible outer vertical morphisms
\begin{equation}\label{eq:MV111}
\xymatrix@C=0.9em@R=2em{
E(X;Z \cap W) \ar[d]_-\simeq \ar[r] & E(X;Z) \ar[d]_-{E(p_2^\ast)} \ar[r]^-{E(p_1^\ast)} & E(V_1;Z_1) \ar[d] \ar[r]^-\partial & \Sigma E(X; Z \cap W) \ar[d]_-\simeq \\
E(V_2; Z_2 \cap p_2^{-1}(W))\ar[r] & E(V_2; Z_2) \ar[r] & E(V_{12}; Z_{12}) \ar[r]^-\partial & \Sigma E(V_2; Z_2\cap p_2^{-1}(W))\,.
}
\end{equation}
By applying the homological functor $\Hom_\cT(b,-)$ to the commutative diagram \eqref{eq:MV111}, for every object $b$ of $\cT$, we observe that the middle square forms a ``Mayer-Vietoris'' LES-triangle \eqref{eq:Nisnevich-1}
with  boundary morphism $\delta$ induced by the composition
$ E(V_{12};Z_{12}) \stackrel{\partial}{\to} \Sigma E(V_2;Z_2\cap p_2^{-1}(W)) \simeq \Sigma E(X;Z\cap W)$.
\end{proof}

\begin{remark}[Generalization]\label{rk:Drinfeld}
  Given a dg category $\cA$, Drinfeld proved in
  \cite[Prop.~1.6.3]{Drinfeld} that the functor $-\otimes \cA$
  preserves short exact sequences of dg categories. By tensoring \eqref{eq:perfect1} with $\cA$, we obtain an induced LES-triangle 
\[
E(X;Z;\cA) \stackrel{\pm}{\to} E(V_1;Z_1;\cA)\oplus E(V_2;Z_2;\cA) \to E(V_{12};Z_{12};\cA) \stackrel{\delta}{\to} \Sigma E(X;Z;\cA)\,.
\]
\end{remark}
%-------------------------------------------------------------------------------
\section{Proof of Theorem \ref{thm:main}}\label{sec:thm}
%-------------------------------------------------------------------------------
%Let $\perf_\dg(X)_Z$ be the full dg subcategory of $\perf_\dg(X)$ consisting of those perfect complexes of $\cO_X$-modules that are supported on $Z$. 
Thanks to the work of Thomason-Trobaugh \cite[\S5]{TT} (see also Proposition \ref{prop:SES}), we have the following short exact sequence of dg categories
\begin{equation}\label{eq:ses-classical}
0 \too \perf_\dg(X)_Z \too \perf_\dg(X) \stackrel{j^\ast}{\too} \perf_\dg(U) \too 0\,.
\end{equation}
Consequently, we obtain an induced distinguished triangle 
$$E(X;Z) \too E(X) \stackrel{E(j^\ast)}{\too} E(U) \stackrel{\partial}{\too} \Sigma E(X;Z)\,.$$ 
Since the dg functor $i_\ast\colon \perf_\dg(Z) \to \perf_\dg(X)$ factors through the inclusion $\perf_\dg(X)_Z \subset \perf_\dg(X)$, we have also an induced morphism
\begin{equation}\label{eq:induced}
E(i_\ast)\colon E(Z) \too E(X;Z)\,.
\end{equation}
The proof of Theorem \ref{thm:main} follows now from the following result:
\begin{theorem}[``D\'evissage'']\label{thm:devissage} 
The morphism \eqref{eq:induced} is invertible.
\end{theorem}
The proof of Theorem \ref{thm:devissage} is divided into three main steps:
\begin{itemize}
\item[(i)] Firstly, we describe the behavior of an $\bbA^1$-homotopy invariant with respect to $\bbN_0$-graded dg categories;
\item[(ii)] Secondly, by combining the first step with a formality result of independent interest (see Theorem \ref{thm:formality}), we prove Theorem \ref{thm:devissage} in the affine case;
\item[(iii)] Thirdly, using the Nisnevich descent results established in \S\ref{sec:Nisnevich}, we bootstrap the proof of Theorem \ref{thm:devissage} from the affine case to the general case.
\end{itemize}
%-------------------------------------------------------------------------------
\subsection*{Step I: Gradings}
%-------------------------------------------------------------------------------
A dg category $\cA$ is called {\em $\bbN_0$-graded} if the (cochain) complexes of $k$-vector spaces $\cA(x,y)$ are equipped with a direct sum decomposition $\bigoplus_{n\geq 0} \cA(x,y)_n$ which is preserved by the composition law. The elements of $\cA(x,y)_n$ are called of {\em pure} degree $n$. Let $\cA_0$ be the dg category with the same objects as $\cA$ and $\cA_0(x,y):=\cA(x,y)_0$. Note that we have an ``inclusion'' dg functor $\iota_0\colon \cA_0 \to \cA$ and a ``projection'' dg functor $\pi\colon \cA \to \cA_0$ such that $\pi \circ \iota_0 =\id$.
\begin{remark}\label{rk:pure}
Let $\cA$ be a dg category whose (cochain) complexes of $k$-vector spaces $\cA(x,y)$ have zero differential and are supported in non-positive degrees. In this case, the dg category $\cA$ becomes $\bbN_0$-graded: an element of $\cA(x,y)$ is of pure degree $n$ if it is of cohomological degree $-n$.
\end{remark}
\begin{remark}\label{rk:tensor}
The tensor product of a $\bbN_0$-graded dg category $\cA$ with a dg category $\cB$ is again a $\bbN_0$-graded dg category with $(\cA\otimes \cB)((x,w),(y,z))_n:=\cA(x,y)_n\otimes \cB(w,z)$.
\end{remark}
The following result is classical.
\begin{lemma}[see \cite{Hoobler}]\label{lem:grading} 
For every $\bbA^1$-homotopy invariant $E\colon\dgcat(k) \to \cT$ and $\bbN_0$-graded dg category $\cA$, we have an associated isomorphism $E(\iota_0)\colon E(\cA_0)\to E(\cA)$.
\end{lemma}
\begin{proof}
Since $\pi\circ \iota_0=\id$, it suffices to show that $E(\iota_0\circ \pi)=\id$. Note that we have canonical dg functors 
$\iota \colon \cA \to \cA[t]$ and $\mathrm{ev}_0, \mathrm{ev}_1: \cA[t] \to \cA$ verifying the equalities $\mathrm{ev}_0 \circ \iota = \mathrm{ev}_1 \circ \iota = \id$. Consider the following commutative diagram
\begin{equation}\label{eq:triangle-diagram}
\xymatrix{
&& \cA \\
\cA \ar@/_1.0pc/@{=}[drr] \ar[rr]^-H  \ar@/^1.0pc/[urr]^-{\iota_0 \circ \pi} && \cA[t] \ar[u]_-{\mathrm{ev}_0} \ar[d]^-{\mathrm{ev}_1} \\
&& \cA\,,
}
\end{equation}
where $H$ is the dg functor defined by $x \mapsto x$ and $\cA(x,y)_n \to \cA(x,y)[t], f\mapsto f\otimes t^n$. Since the functor~$E$ inverts the morphism $\iota$, it also inverts the morphisms $\mathrm{ev}_0$ and $\mathrm{ev}_1$. Moreover, $E(\mathrm{ev}_0)=E(\mathrm{ev}_1)$. By applying the functor $E$ to \eqref{eq:triangle-diagram}, we hence conclude that $E(\iota_0 \circ \pi) = \id$. %This achieves the proof.
\end{proof}
%-------------------------------------------------------------------------------
\subsection*{Step II: Affine case}
%-------------------------------------------------------------------------------
Our main result in the affine case is the following:
\begin{theorem}[Formality]\label{thm:formality}  Let $Z\hookrightarrow X$ a closed immersion of smooth affine $k$-schemes and $\cI \subset \cO_X$ the defining ideal of $Z$ in
  $X$. Then, the following holds:
\begin{itemize}
\item[(i)] The sheaf $\cO_X/\cI \in \perf(X)_Z$ is a compact generator of $\cD_{\mathrm{Qcoh}}(X)_Z$. Consequently, the dg category $\perf_\dg(X)_Z$ is Morita equivalent to $\perf_\dg(A)$, where $A$ stands for the derived dg algebra of endomorphisms ${\bf R}\mathrm{End}_X(\cO_X/\cI)$.
\item[(ii)] The dg algebra $A$ is formal.
\item[(iii)] We have an isomorphism $H^{\ast}(A) \simeq \Gamma(Z, \bigwedge^\ast_Z((\cI/\cI^2)^\vee))$, where $\cI/\cI^2$ is considered as a vector bundle on $Z$.
\end{itemize}
\end{theorem}
\begin{proof}
(i) Given an object $\cF \in \cD_{\mathrm{Qcoh}}(X)_Z$, we need to show the implication:
\begin{eqnarray}\label{eq:implication}
\Hom_{\cD_{\mathrm{Qcoh}}(X)_Z}(\cO_X/\cI, \Sigma^n \cF)=0\quad\forall n \in \bbZ & \Rightarrow & \cF=0\,.
\end{eqnarray}
Since $X$ is affine, the left-hand side of \eqref{eq:implication} is
equivalent to $\underline{\Hom}_X(\cO_X/\cI, \cF)=0$. As proved by Grothendieck in
\cite[Prop.~(19.1.1)]{EGA}, the ideal $\cI$ is locally generated by a
regular sequence. Given an (affine) open subscheme $U \hookrightarrow
X$ such that $Z \cap U =V(f_1, \ldots, f_d)$ for a regular sequence
$f_1, \ldots, f_d \in \Gamma(U,\cO_U)$, the object $(\cO_X/\cI)_{|U}:=
\bigotimes_i \mathrm{cone}(f_i\colon \cO_U\to\cO_U)$ is well-known to
be a compact generator of $\cD_{\mathrm{Qcoh}}(U)_{Z\cap U}$; see \cite[Prop.~6.1]{Neeman}. Using the
natural identification between $\underline{\Hom}_X((\cO_X/\cI)_{|U},
\cF_{|U})$ and $\underline{\Hom}_X(\cO_X/\cI,\cF)_{|U}$, we 
conclude that $\cF_{|U}=0$. The proof follows now from the fact that
$X$ admits a covering by such (affine) open subschemes $U$'s.

(ii) It is now be convenient to switch to ring theoretical notation: let $X=\mathrm{Spec}(R)$,  $Z=\mathrm{Spec}(S)$, and  $\phi\colon R \twoheadrightarrow S$ the surjective $k$-algebra homomorphism corresponding to the closed immersion $Z \hookrightarrow X$.
Let us write $I$ for the kernel of $\phi$, $\widehat{R}$ for the completion of $R$ at $I$, $T$ for the graded symmetric algebra $\mathrm{Sym}_S(I/I^2)$, and $\widehat{T}$ for the completion of $T$ at the augmentation ideal $T_{>0}$. 
Thanks to Lemma \ref{lem:tubular} below, $I/I^2$ is a finitely generated projective $S=R/I$-module and
there exists an isomorphism $\widehat{T} \stackrel{\simeq}{\to} \widehat{R}$ compatible with $\phi$ and with the projection $T\rightarrow T_0=S$.

Let us write $\cD(R)$ for the derived category of $R$ and $\cD(R)_I$
for the full triangulated subcategory of those complexes of
$R$-modules whose cohomology is locally annihilated by a
power of $I$. Note that $\cD(R)_I$ is equivalent to
$\cD_{\mathrm{Qcoh}}(X)_Z$. We have the following equivalence of categories
\begin{equation}\label{eq:equivalences}
\cD_{\mathrm{Qcoh}}(X)_Z \simeq \cD(R)_I \stackrel{(a)}{\simeq} \cD(\widehat{R})_{\widehat{I}} \stackrel{(b)}{\simeq} \cD(\widehat{T})_{\widehat{T}_{\geq 1}} \stackrel{(c)}{\simeq} \cD(T)_{T_{\geq 1}}\,,
\end{equation}
where (a) and (c) follow from Theorem \ref{th:tt} and (b) from Proposition
\ref{prop:tubular}. Via the equivalences \eqref{eq:equivalences},
the sheaf $\cO_X/I \in \cD_{\mathrm{Qcoh}}(X)_Z$ corresponds to the
$T$-module $S \in \cD(T)_{T_{\geq 1}}$. Consequently, the dg algebra
$A$ becomes quasi-equivalent to the derived dg algebra of
endomorphisms $A':={\bf R}\mathrm{End}_T(S)$. The formality of $A'$ 
follows now from Proposition \ref{lem:koszul} below.

(iii) The proof follows from Proposition \ref{lem:koszul} below.
\end{proof}
%We have used the following standard result.
\begin{proposition}[Affine tubular neighborhoods]\label{prop:tubular}
\label{lem:tubular} Let $\phi:R\twoheadrightarrow S$ be a surjective morphism between smooth $k$-algebras, with kernel $I$.
 Then, the $S=R/I$-module $I/I^2$ is finitely generated projective. Let $\widehat{R}$ be the $I$-adic completion of $R$ at $I$, $T$ the graded symmetric algebra $\mathrm{Sym}_S(I/I^2)$ (with $T_n:=\mathrm{Sym}^n_S(I/I^2)$), and $\widehat{T}$ the completion of $T$ at the ideal $T_{\geq 1}$. Under these notations, there exists an isomorphism $\tau\colon \widehat{T} \stackrel{\simeq}{\to} \widehat{R}$ such that $\widehat{\phi}\circ \tau$ agrees with the projection onto $T_0=S$.
\end{proposition}
\begin{proof} 
Since the $k$-algebras $R$ and $S$ are formally smooth (see \cite[Prop.~E.2]{Loday}), the proof follows from \cite[Cor.\ 0.19.5.4]{EGA} applied to  $A:=k$, $B:=\widehat{R}$, and $C:=S$. The property that $\widehat{\phi}\circ \tau$ agrees
  with the projection onto $T_0=S$ is not explicitly stated in {\em loc. cit.}, but it follows immediately from the proof.
%  Since $R,S$ are formally smooth (see \cite[Prop.~E.2]{Loday}) this
%  is a special case of \cite[Cor.\ 0.19.5.4]{EGA} with $A=k$, $B=\hat{R}$,
%  $C=S$. The property that $\widehat{\phi}\circ \tau$ agrees
%  with the projection onto $T_0=S$ is not explicitly stated in loc.\
%  cit.\ but it follows immediately from the proof.
\end{proof}
%We have also used the following basic result 
\begin{proposition}[Koszul duality] \label{lem:koszul}
Let $S$ be a commutative ring, $P$ a finitely generated projective $S$-module, and $T:=\Sym_S(P)$.
Then, the dg algebra $A:=\REnd_T(S)$ is formal and its cohomology as a graded algebra is given by $\bigwedge^\ast_S P^\vee$.
\end{proposition}
\def\Ext{\operatorname{Ext}}
\begin{proof} We may compute the cohomology $H^\ast(A)=\Ext^\ast_T(S,S)$ via the standard Koszul resolution of $S$ given by $\mathbf{K}:=(\Sym\nolimits_S (\Sigma P)\otimes_S T,\eta\cap-)$, where $\eta$ is the unit element in $P^\vee \otimes_S P$.
Since the differential in $\mathbf{K}$ becomes zero after applying $\Hom_T(-,S)$, we have abelian group isomorphisms
\begin{equation}
\label{eq:ext}
\Ext^\ast_T(S,S)\simeq H^\ast(\Hom_T(\mathbf{K},S))\simeq\Sym\nolimits_S \Sigma^{-1}( P^\vee)\,.
\end{equation}
Next, we observe that if $\omega\in \Sym^n_S \Sigma^{-1}(P^\vee)$, then $\omega\cap-$ (super-)commutes with $\eta\cap-$. This implies that $\omega\cap-$ defines a morphism of complexes $\mathbf{K}\r \Sigma^n \mathbf{K}$.
We obtain in this way a dg algebra morphism $\Sym_S \Sigma^{-1}(P^\vee)\r \End_T(\mathbf{K}), \omega\mapsto \omega\cap-$, such that
the composition $\Sym_S \Sigma^{-1}(P^\vee)\r H^\ast(\End_T(\mathbf{K}))\r H^\ast(\Hom_T(\mathbf{K},S))$ is
the inverse of \eqref{eq:ext}. This concludes the proof.
%Let $S$ be a commutative ring, $P$ a finitely generated projective $S$-module, and $T:=\Sym_S(P)$.
%Then, the dg algebra $A:=\REnd_T(S)$ is formal and its cohomology as a graded algebra is given by $\bigwedge^\ast_S P^\vee$.
%\end{lemma}
%\def\Ext{\operatorname{Ext}}
%\begin{proof} We may compute $H^\ast(A)=\Ext^\ast_T(S,S)$ via the standard Koszul resolution of $S$
%which is given by $\mathbf{K}:=(\Sym\nolimits_S (\Sigma P)\otimes_S T,\eta\cap-)$
%where $\eta$ is the unit element in $P^\ast \otimes_S P$.
%Since the differential in $\mathbf{K}$ becomes zero after applying $\Hom_T(-,S)$ we
%obtain as abelian groups
%\begin{equation}
%\label{eq:ext}
%\Ext^\ast_T(S,S)=H^\ast(\Hom_T(\mathbf{K},S))=\Sym\nolimits_S \Sigma^{-1}( P^\vee)
%\end{equation}
%Next we observe that if $\omega\in \Sym^k_S \Sigma^{-1}(P^\vee)$ then $\omega\cap-$ (super)commutes with $\eta\cap-$
%and hence $\omega\cap-$ defines a morphism of complexes $\mathbf{K}\r \Sigma^k \mathbf{K}$.
%In this way we obtain a DG-algebra morphism $\Sym_S \Sigma^{-1}(P^\ast)\r \End_T(\mathbf{K}):\omega\mapsto \omega\cap-$ such that
%the composition $\Sym_S \Sigma^{-1}(P^\vee)\r H^\ast(\End_T(\mathbf{K}))\r H^\ast(\Hom_T(\mathbf{K},S))=\Ext^\ast_T(S,S)$ is
%the inverse of \eqref{eq:ext}. This finishes the proof. 
%\marginpar{\Michel{Gon\c calo: I deleted the example. It is a really basic fact that the (dg-)Koszul dual of $k[t]$ is $k[\epsilon]/(\epsilon^2)$...}}
\end{proof}
%\begin{example} 
%Let $X=\mathrm{Spec}(k[t])$ be the affine line and $Z$ the closed point $t=0$. In this particular case, making use of the quasi-isomorphism between $k[t]/tk[t]\simeq k$ and $\mathrm{cone}(t\colon k[t] \to %k[t])$, the dg algebra $A$ can be explicitly described as follows 
%\begin{equation*}
%\cdots \too 0 \too k[t] \stackrel{d_{-1}}{\too} k[t] \oplus k[t] \stackrel{d_0}{\too} k[t] \too 0 \too \cdots\,,
%\end{equation*}
%where $d_{-1}(p(t)):=(p(t) t, t p(t))$ and $d_0(p(t),q(t)):=t p(t) - q(t)  t$. The multiplication law is induced by composition. Since $tk[t]/(tk[t])^2\simeq k$, $H^\ast(A)$ identifies with the algebra of %dual numbers $k[\epsilon]/\epsilon^2$ where $\epsilon$ is of cohomological degree $-1$.
%\end{example}
We may now conclude the proof of Theorem \ref{thm:main} in the affine
case. Thanks to Theorem \ref{thm:formality}, the dg category
$\perf_\dg(X)_Z$ is Morita equivalent to $H^\ast(A)$ and the dg
functor $i_\ast\colon \perf_\dg(Z) \to \perf_\dg(X)_Z$ identifies with
the inclusion of $H^0(A)$ into $H^\ast(A)$. Using
Lemma \ref{lem:grading} and  Remark \ref{rk:pure}, we hence conclude
that the induced morphism \eqref{eq:induced} is invertible.
\subsection*{Step III: General case}
%-------------------------------------------------------------------------------
In order to bootstrap the proof of Theorem \ref{thm:devissage} from
the affine case to the general case, we use  the following
induction principle:
\begin{proposition}{(see \cite[Prop.~3.3.1]{BV})}\label{prop:reduction}
Given a property $\cP$, assume the following:
\begin{itemize}
\item[(A1)] The property $\cP$ holds for all affine $k$-schemes $X$;
\item[(A2)] Let $V_1 \cup V_2=X$ be a Zariski open cover of a scheme $X$; 
such that $X,V_1,V_2,V_{12}$ are quasi-compact quasi-separated
(see \S\ref{sec:Nisnevich}). If the property $\cP$ holds for $V_1$, $V_2$, and $V_{12}$, then it also holds for $X$.
\end{itemize}
Under the assumptions (A1)-(A2), the property $\cP$ holds for all quasi-compact quasi-separated $k$-schemes.
\end{proposition}
Let $\cP$ be the following property: ``If $X$ is a smooth $k$-scheme, then \eqref{eq:induced} is invertible for every smooth closed subscheme $Z$''.
 As proved in Step II, the assumption (A1) of Proposition \ref{prop:reduction} is satisfied. 
Let us now verify assumption (A2). Given a smooth $k$-scheme $X$ and a smooth closed subscheme $i\colon Z \hookrightarrow X$, consider the following commutative diagram
%In order to prove assumption (A2) we need then to show that if by hypothesis the morphisms
%\begin{eqnarray*}
%E(\perf_\dg(Z_i)) \too E(\perf_\dg(X_i)_{Z_i}) && E(\perf_\dg(Z_{12})) \too E(\perf_\dg(X_{12})_{Z_{12}}) 
%\end{eqnarray*}
%are invertible, then $E(i_\ast)\colon E(\perf_\dg(Z)) \to E(\perf_\dg(X)_Z)$ is also invertible. Consider the following commutative diagram:
\begin{equation}\label{eq:bigsquares}
\xymatrix@C=2.5em@R=2em{
E(X;Z) \ar[rrr]^-{E(p_1^\ast)} \ar[ddd]_-{E(p_2^\ast)}& & & E(V_1; Z_1) \ar[ddd] \\
& E(Z) \ar[ul]_-{E(i_\ast)}  \ar[r]^-{E(p_1^\ast)} \ar[d]_-{E(p_2^\ast)}& E(Z_1) \ar[ur]^-{E(i^1_\ast)} \ar[d] & \\
& E(Z_2) \ar[dl]_-{E(i^2_\ast)} \ar[r] & E(Z_{12}) \ar[dr]^-{E(i^{12}_\ast)} & \\
E(V_2; Z_2) \ar[rrr] & & & E(V_{12}; Z_{12})
}
\end{equation}
in the triangulated category $\cT$, where $Z_i:=Z\cap V_i$ and
$Z_{12}:=Z_1 \cap Z_2$.  In order to prove assumption (A2) we need
 to show that if the morphisms $E(i^1_\ast)$,
$E(i^2_\ast)$, and $E(i^{12}_\ast)$ are invertible, then $E(i_\ast)$
is also invertible. Thanks to
Theorem \ref{thm:Nisnevich}, \eqref{eq:bigsquares} gives rise to the following morphism between LES-triangles:
%(using the fact that $E(Z;Z)=E(Z)$, etc\dots)
%
% resp. Remark \ref{rk:nosupport1}, the
%outer, resp. inner, commutative square in \eqref{eq:bigsquares} leads
%to the Mayer-Vietoris long exact sequence \eqref{eq:MV},
%resp. \eqref{eq:MV1} (with $X=Z$, $V_i=Z_i$, and
%$V_{12}=Z_{12}$). Consequently, the commutative diagram
%\eqref{eq:bigsquares} gives rise to the following morphism between
%Mayer-Vietoris long exact sequences:
$$
\xymatrix@C=2em@R=2.5em{
  E(Z) \ar[d]_-{E(i_\ast)} \ar[r]^-\pm & E(Z_1)\oplus E(Z_2) \ar[d]_-{E(i^1_\ast)\oplus E(i^2_\ast)}\ar[d]^{\simeq} \ar[r] & E(Z_{12}) \ar[d]_-{E(i^{12}_\ast)}\ar[d]^{\simeq} \ar[r]^-{\delta} & \Sigma E(Z) \ar[d]_-{E(i_\ast)}  \\
  E(X;Z) \ar[r]_-\pm & E(V_1;Z_1) \oplus E(V_2;Z_2) \ar[r] &
  E(V_{12};Z_{12}) \ar[r]_-{\delta} & \Sigma E(X;Z)\,. }
$$
Making use of the $5$-lemma, we conclude that $E(i_\ast)$ becomes invertible after applying $\Hom_{\cT}(b,-)$, for every object $b$ of $\cT$. The Yoneda lemma hence implies that $E(i_\ast)$ is invertible. This finishes the proof.
\section{Proof of the generalizations (G1)-(G2)}
%-------------------------------------------------------------------------------
%-------------------------------------------------------------------------------
\subsection*{Generalization (G1)}\label{sub:G1}
%-------------------------------------------------------------------------------
Consider the cartesian square of algebraic spaces 
\begin{equation}\label{eq:square-spaces}
\xymatrix{
\cV_{12}:=\cV_1 \times_{\cX} \cV_2 \ar[d] \ar[r] & \cV_2 \ar[d]^-{p_2} \\
\cV_1 \ar[r]_-{p_1} & \cX\,,
}
\end{equation}
where $p_1$ is an open immersion and $p_2$ is an {\'e}tale map inducing an isomorphism on reduced algebraic spaces $p_2^{-1}(\cX\backslash \cV_1)_{\mathrm{red}}\simeq (\cX\backslash \cV_1)_{\mathrm{red}}$. Let $i\colon \cZ \hookrightarrow \cX$ be a closed algebraic subspace, $\cZ_1:=\cZ\cap \cV_1$, $\cZ_2:=p_2^{-1}(\cZ)$, and $\cZ_{12}:=\cZ_1 \times_\cZ \cZ_2$.
\begin{remark}\label{rk:alg-spaces}
Theorem \ref{thm:Nisnevich} holds similarly with \eqref{eq:Nisnevich} replaced by \eqref{eq:square-spaces} and $Z_1$, $Z_2$, and $Z_{12}$, replaced by $\cZ_1$, $\cZ_2$, and $\cZ_{12}$, respectively. %\Gon\c calo{[Michel: could you kindly prove these claims or give a precise reference?]}
\end{remark}
In the case of algebraic spaces, Proposition \ref{prop:reduction} admits the following variant:
\begin{proposition}{(see \cite[Tag 08GL, Lem.~62.8.3]{Stacks})}\label{prop:reduction2}
Given a property $\cP$, assume:
\begin{itemize}
\item[(A1)] The property $\cP$ holds for all affine $k$-schemes $\cX$;
\item[(A2)] If the property $\cP$ holds for the quasi-compact quasi-separated algebraic spaces $\cV_1$, $\cV_2$, $\cV_{12}$ in \eqref{eq:square-spaces}, with  $\cV_2$ being an affine $k$-scheme, then it also holds~for~$\cX$.
\end{itemize}
Under the assumptions (A1)-(A2), the property $\cP$ holds for all quasi-compact quasi-separated algebraic spaces.
\end{proposition}
Let $\cX$ be a smooth algebraic space, $i\colon \cZ \hookrightarrow \cX$ a smooth closed algebraic space, and $j\colon \cU \hookrightarrow \cX$ the open complement of $\cZ$. Similarly to \eqref{eq:ses-classical}, we have the following short exact sequence of dg categories %\Gon\c calo{[Michel: could you kindly prove or add a reference for this claim?]}
$$ 0 \too \perf_\dg(\cX)_\cZ \too \perf_\dg(\cX) \stackrel{j^\ast}{\too} \perf_\dg(\cU)\too 0\,.$$
Therefore, in order to establish (G1), it suffices to prove the
generalization of Theorem \ref{thm:devissage} obtained by replacing
\eqref{eq:induced} with $E(i_\ast)\colon E(\cZ) \to E(\cX;\cZ)$. Steps
I-II hold {\em mutatis mutandis}. For Step III, simply
replace Proposition \ref{prop:reduction} by Proposition
\ref{prop:reduction2} and run the same argument using Remark
\ref{rk:alg-spaces}.
%-------------------------------------------------------------------------------
\subsection*{Generalization (G2)}\label{sub:G2}
%-------------------------------------------------------------------------------
Given a dg category $\cA$, Drinfeld proved in \cite[Prop.~1.6.3]{Drinfeld} that the functor $-\otimes \cA$ preserves short exact sequences of dg categories. Consequently, in order to establish 
(G2), it suffices to prove the generalization of Theorem \ref{thm:devissage} obtained by replacing \eqref{eq:induced} with the morphism
\begin{equation}\label{eq:morphism-generalized}
 E(i_\ast \otimes \id) \colon E(Z;\cA) \too E(X;Z;\cA)\,.
 \end{equation}
Step I holds {\em mutatis mutandis}. For Step II, recall from Theorem \ref{thm:formality} that the dg category $\perf_\dg(X)_Z$ is Morita equivalent to $H^\ast(A)$ and that the dg functor $i_\ast\colon \perf_\dg(X) \to \perf_\dg(X)_Z$ identifies with the inclusion of $H^0(A)$ into $H^\ast(A)$. Thanks to Remarks \ref{rk:pure}-\ref{rk:tensor}, we hence conclude from Lemma \ref{lem:grading} that the morphism \eqref{eq:morphism-generalized} is also invertible. For Step III, run the same argument using Remark \ref{rk:Drinfeld}. Finally, the second claim of Generalization (G2) follows from  Lemma \ref{lem:product}.
%-------------------------------------------------------------------------------
\section{Proof of Theorem \ref{thm:main2}}\label{sec:proofmain2}
%-------------------------------------------------------------------------------
%In order to simplify the exposition, let us write $\cT^{\mathrm{sp}}$ (resp. $\cT^{\mathrm{sp}}_{\bbZ[S^{-1}]}$) for the smallest triangulated (resp. thick triangulated) subcategory of $\cT$ (resp. $\cT_{\bbZ[S^{-1}]}$) containing the objects $E(\perf_\dg(Y))$ (resp. $E(\perf_\dg(Y))_{\bbZ[S^{-1}]}$) with $Y$ a smooth projective $k$-scheme. 
We start with the following  birationality result:
\begin{proposition}\label{prop:birational}
Let $k$ be a perfect field, $X$ and $Y$ two birational smooth connected $k$-schemes, and $E\colon \dgcat(k) \to \cT$ a functor which satisfies conditions (C1)-(C2). 
Let $\cU$ be a triangulated subcategory of $\cT$.
Assume that all the objects $E(W)$, with~$W$ a smooth $k$-scheme of dimension strictly inferior to $\mathrm{dim}(X)$, belong to $\cU$.
Under these assumptions, if $E(X)$ or $E(Y)$ belongs to $\cU$, so it does the other one.
\end{proposition}
\begin{proof}
It is sufficient to consider the case where $Y$ is an open subscheme of $X$. Let $Z \hookrightarrow X$ be the closed complement of $Y$. Since by assumption $k$ is a perfect field, there exists a stratification of $Z$ into closed subschemes
$$ \varnothing =Z_{-1} \hookrightarrow Z_0 \hookrightarrow \cdots \hookrightarrow Z_r \hookrightarrow \cdots \hookrightarrow Z_{n-1} \hookrightarrow Z_n =Z$$
such that $Z_r \backslash Z_{r-1}$ is smooth for every $0 \leq r \leq n$. Consider the Gysin triangles
\begin{equation}\label{eq:Gysin-several}
E(Z_r \backslash Z_{r-1}) \too E(X\backslash Z_{r-1}) \too E(X\backslash Z_r)\stackrel{\partial}{\too} \Sigma E(Z_r \backslash Z_{r-1})
\end{equation}
provided by Theorem \ref{thm:main}. Since $Y$ is an open {\em dense} subscheme of $X$, the dimension of $Z_r \backslash Z_{r-1}$ is strictly inferior to $\mathrm{dim}(X)$ for every $0 \leq r \leq n$. Therefore, the proof follows from now recursively from the Gysin triangles \eqref{eq:Gysin-several}.
\end{proof}

%-------------------------------------------------------------------------------
\subsection*{Item (i)}
%-------------------------------------------------------------------------------
Let $\cU$ be the triangulated category $\cT^{\mathrm{sp}}$. Without loss of generality we may assume that
$X$ is connected. Furthermore, using induction on $\mathrm{dim}(X)$ we
may assume that all the objects $E(W)$, with $W$ a smooth $k$-scheme
of dimension strictly inferior to $\mathrm{dim}(X)$, belong to
$\cU$.  Since by
assumption $k$ admits resolution of singularities, $X$ can be realized
as the open complement of a strict normal crossing divisor $D$ inside
a smooth projective $k$-scheme $Y$. Using the fact that $E(Y)$ belongs to $\cU$, the proof follows now from Proposition \ref{prop:birational}.
\begin{remark}
  Let $D_i$ be the irreducible components of $D$. One may easily show that the object $E(X)$ belongs to the smallest triangulated
  subcategory of $\cT$ containing the objects $E(Y)$ and
  $\{E(D_{r_1} \cap \cdots \cap D_{r_i})\,|\, 1 \leq r_1, \ldots, r_i
  \leq m\}$.
\end{remark}
%-------------------------------------------------------------------------------
\subsection*{Item (ii)}
%-------------------------------------------------------------------------------
%\begin{remark}\label{rk:birational}
%Clearly, Proposition \ref{prop:birational} holds similarly with $\cT^{\mathrm{sp}}_\bbQ$ replaced by $\cT^{\mathrm{sp}}$.
%\end{remark}

Let $\cU$ be the triangulated category $\overline{\cT^{\mathrm{sp}}}$. Without loss of generality we may assume that $X$ is connected. Furthermore by induction on  $\mathrm{dim}(X)$ we may assume that all the objects $E(W)$, with $W$ a smooth $k$-scheme of dimension strictly inferior to $\mathrm{dim}(X)$, belong to $\cU$.
We start with the following result:
\begin{proposition}
\label{lem:alt}
For each prime $l\neq p$, there exists an open dense subscheme $V\hookrightarrow X$ and a finite \'etale cover $g_V:V'\r V$ such that $g_{V\ast}(\cO_{V'})\simeq \cO_V^{\oplus d}$ with $d$ coprime to $l$. Moreover, $E(V')$ belongs to $\cU$.
\end{proposition}
\begin{proof}
Gabber's refined version of
de Jong's theory  \cite{Jong} of alterations (see \cite[Thms. 3(i) and 3.2.1]{Laszlo}) allows us to construct for each prime $l\neq p$ a diagram
\begin{equation*}%\label{eq:diagram-schemes}
\xymatrix{
V':=V \times_X X' \ar[d]_-{g_V} \ar[r]^-{j'}& X':=X\times_{\overline{X}} Y \ar[d] \ar[r] & Y \ar[d]^-{g} \\
V \ar[r]_-j & X \ar[r] & \overline{X}\,,
}
\end{equation*}
where $\overline{X}$ is a compactification of $X$, $g$ is an
alteration, $j\colon V \hookrightarrow X$ is an open dense subscheme,
and $g_V$ is a finite {\'e}tale surjective map of rank $d$ prime to
$l$.  Shrinking $V$ if necessary, we may assume that
$(g_V)_\ast(\cO_{V'})\simeq \cO_V^{\oplus d}$. Since $Y$ is irreducible, the open subscheme $V'\hookrightarrow Y$ is
dense. Using the fact that $Y$ is smooth projective and
$\mathrm{dim}(Y)=\mathrm{dim}(X)$, we conclude from Proposition~\ref{prop:birational} that $E(V')$ belongs to~$\cU$.
\end{proof}
Choose an arbitrary prime $l_0\neq p$ and construct $V,V',g_V,d$ as
in Proposition \ref{lem:alt} with $l=l_0$. Denote the result by
$V'_0,V_0,g_{V_0},d_0$.  Let $\{l_1,\ldots,l_t\}$ be the prime factors
of $d_0$ distinct from $p$. Construct $V',V,g_V,d$ for each
$l=l_i$ and denote the result by
$V'_i,V_i,g_{V_i},d_i$. Let us write $V$ for the intersection $\bigcap_{i=1}^t V_i$. Thanks to Proposition \ref{prop:birational}, we can assume that $V_0=V_1=\cdots=V_t=V$ and replace $V'_i$ by $g^{-1}_{V_i}(V)$ without changing the statement of Proposition \ref{lem:alt}; in particular, $E(V'_i)$ belongs to $\cU$. Lemma \ref{prop:retractions} below hence implies that the composition $E(g_{V_i,\ast})\circ E(g_{V_i}^\ast)\colon E(V) \to E(V'_i) \to E(V)$ is equal to $d_i$ times the identity.

Now, choose integers $(a_i)_i\in \bbZ$ such that $\sum_i a_i d_i=\mathrm{gcd}(d_0, \ldots, d_t)$. Under these choices, the following (matrix) composition
\[
E(V)\xrightarrow{[E(g_{V_i}^\ast)]_{t \times 1}} \bigoplus_{i=0}^t E(V'_i) \xrightarrow{[a_iE(g_{V_i,\ast})]_{1 \times t}} E(V)
\]
is equal to $e:=\mathrm{gcd}(d_0, \ldots, d_n)$ times the identity. We claim that $e$ is a power of $p$. Indeed, if
$q\neq p$ is a prime which divides $e$, then $q$ divides $d_0$. This implies that $q=l_i$ for some $1\leq i\leq t$. But then, we would conclude that $q$ does not divides $d_i$, which is a contradiction! Since $e$ is a power of $p$ and, by assumption, the triangulated category $\cT$ is $\bbZ[1/p]$-linear, the object $E(V)$ is a direct summand of $\bigoplus_{i=0}^t E(V'_i)$. Using the fact that $E(V'_i)$ belongs to $\cU$ and that the triangulated category $\cU$ is idempotent complete, we hence conclude that $E(V)$ also belongs to $\cU$. The proof of Theorem Item (ii) follows now from Proposition \ref{prop:birational} applied to the open dense subscheme $j\colon V\hookrightarrow X$.
%
%Now choose a prime $l_0\neq p$ arbitrary and construct $g_V,V,V',d$ as
%in Lemma \ref{lem:alt} with $l=l_0$. Denote the result by
%$g_{V_0},V'_0,V_0,d_0$.  Let $\{l_1,\ldots,l_t\}$ be the prime factors
%of $d_0$ distinct from $p$ and construct $g_{V'},V',V,d$ for each
%$l=l_i$ for $i=1,\ldots,t$. Denote the results by
%$g_{V_i},V'_i,V_i,d_i$,
%
%Replacing all the $V_i$ by the intersection $V=\cap_i V_i$, replacing $V_i'$ by $g_{V_i}^{-1}(V)$ and invoking Proposition \ref{prop:birational} for the inclusion $g_{V_i}^{-1}(V)\subset V'_i$
%we may assume that $V_0=V_1=\cdots=V_t=V$ while preserving the properties of $g_{V_i}$, $V'_i$ that are stated in Lemma \ref{lem:alt}. Thanks to Proposition \ref{prop:retractions} the composition $E(g_{V_i,\ast})\circ E(g_{V_i}^\ast)\colon E(V) \to E(V'_i) \to E(V)$ is multiplication by $d_i$. Put $e:=\gcd(d_0,\ldots,d_t)$
%and let $(a_i)_i\in \bbZ$ be such that
%$\sum_i a_i d_i=e$.
%It follows that
%\[
%E(V)\xrightarrow{\oplus_i E(g_{V_i}^\ast)} \bigoplus_{i=0}^t E(V'_i) \xrightarrow{\oplus_i a_iE(g_{V_i,\ast})} E(V)
%\]
%is multiplication by $e$. 
%We claim that $e$ is a power of $p$. Indeed if
%$q\neq p$ is a prime which divides $e$ then $q$ divides $d_0$ and hence $q=l_i$ for some $i\ge 1$. But then $q$ does not divide $d_i$, which is a contradiction. Since $e$ is a power of $p$ and $\cT$ is $\bbZ[1/p]$ linear, the object $E(V)$ is a direct summand of $\bigoplus_{i=0}^t E(V'_i)$ and hence is contained in $\cU$.
%Finally, using once again Proposition \ref{prop:birational} for the inclusion
%$V\subset X$, we conclude that $E(X)$ belongs to $\cU$.
\begin{lemma}\label{prop:retractions}
 Let $f\colon X \to Y$ be a finite map between quasi-compact quasi-separated $k$-schemes such that
  $f_\ast(\cO_X)\simeq \cO_Y^{\oplus d}$. Then, for every additive invariant
  $F$, the composition $F(f_\ast)\circ
  F(f^\ast)\colon F(Y) \to F(X) \to F(Y)$ is equal to $d$ times the
  identity.
\end{lemma}
\begin{proof}
Since $F$ factors through the universal additive invariant $\mathrm{U}_{\mathrm{add}}$ (see \S\ref{sec:universal}), it is sufficient
to prove Lemma \ref{prop:retractions} in the case $F=\mathrm{U}_{\mathrm{add}}$. Thanks to the projection formula (see \cite[\S3.17]{TT}), the dg
functor $f_\ast f^\ast\colon \perf_\dg(Y)\to \perf_\dg(Y)$ is given by
$f_\ast(\cO_X)\otimes - $. 
Since by assumption $f_\ast(\cO_X)\simeq \cO_Y^{\oplus
  d}$, the $\perf_{dg}(Y)$-bimodule ${}_{f_\ast f^\ast}\mathrm{B}$ corresponding to  $f_\ast f^\ast$ (see \eqref{eq:bimodule2})
is isomorphic to
$({}_{\id}\mathrm{B})^{\oplus d}$ in the triangulated category
$\rep(\perf_\dg(Y),\perf_\dg(Y))$. Consequently, $[{}_{f_\ast
  f^\ast}\mathrm{B}]=d[{}_{\id}\mathrm{B}]$ in the Grothendieck group
$K_0\rep(\perf_\dg(Y),\perf_\dg(Y))$. The proof follows now from the fact that $[{}_{\id}\mathrm{B}]$ is the identity of the object
$\mathrm{U}_{\mathrm{add}}(Y) \in \Hmo_0(k)$.
\end{proof}

%-------------------------------------------------------------------------------
\section{Proof of Theorem \ref{thm:cellular2}}
%-------------------------------------------------------------------------------
We start with the following ``invariance'' result concerning affine
fibrations:
\begin{proposition}\label{prop:fibrations}
  Let $f\colon X \to Y$ be an affine fibration between quasi-compact quasi-separated $k$-schemes. For every
  functor $E\colon \dgcat(k) \to \cT$ which satisfies conditions
  (C1)-(C2), we have an induced isomorphism $E(f^\ast)\colon E(Y)
  \stackrel{\simeq}{\to} E(X)$.
\end{proposition}
\begin{proof}
  Let us denote by $d$ the relative dimension of $f$. Using an appropriate variant of Proposition
  \ref{prop:reduction}, it suffices to verify the following
  conditions: %\marginpar{\Michel{Gon\c calo: It is not literally Prop.\ \ref{prop:reduction} since an affine fibration need not be   trivial on affines. It is more like our proof of nilpotency for   $\ker (K(X)\xrightarrow{\rk} \underline{\bbZ})$.}}
\begin{itemize}
\item[(A1)] The morphism $E(Y) \to E(Y \times \bbA^d)$, induced by the projection, is invertible;
\item[(A2)] Let $V_1\cup V_2 =Y$ be a Zariski cover of $Y$. If the morphisms $E(f^\ast_{1})$, $E(f_2^\ast)$, and $E(f_{12}^\ast)$ are invertible, then $E(f^\ast)$ is also invertible.
\end{itemize}
By Lemma  \ref{lem:product}, we have an isomorphism $E(Y\times
\bbA^d)\simeq E(Y;\perf_\dg(\bbA^d))$. Therefore,
condition (A1) follows from the Morita equivalence $k[t]^{\otimes
  d}\to \perf_\dg(\bbA^d)$ and from an iterated application of
condition (C2). In order to prove condition (A2), consider the
following commutative diagram
\begin{equation*}\label{eq:bigsquare22}
\xymatrix@C=2.7em@R=2.5em{
E(X) \ar[rrr]^-{E(p_1^\ast)} \ar[ddd]_-{E(p_2^\ast)}& & & E(X_1) \ar[ddd] \\
& E(Y) \ar[ul]_-{E(f^\ast)}  \ar[r]^-{E(p_1^\ast)} \ar[d]_-{E(p_2^\ast)}& E(V_1) \ar[ur]^-{E(f^\ast_1)} \ar[d] & \\
& E(V_2) \ar[dl]_-{E(f_2^\ast)} \ar[r] & E(V_{12}) \ar[dr]^-{E(f^\ast_{12})} & \\
E(X_2) \ar[rrr] & & & E(X_{12})
}
\end{equation*}
in the triangulated category $\cT$, where $X_i:=f^{-1}(V_i)$ and
$X_{12}:=f^{-1}(V_{12})$.  Thanks to Theorem \ref{eq:Nisnevich}, the outer and
inner squares give rise to  ``Mayer-Vietoris'' LES-triangles. Hence, a proof similar to the one of (Step III: General case)
shows that if $E(f^\ast_{1})$, $E(f_2^\ast)$, and $E(f_{12}^\ast)$ are invertible,
then $E(f^\ast)$ is also invertible.
\end{proof}
The filtration \eqref{eq:filtration}, combined with the isomorphisms $\mathrm{U}(Y_i) \simeq \mathrm{U}(X_i\backslash X_{i-1})$ provided by Proposition \ref{prop:fibrations}, gives rise to the following Gysin triangles in $\Mot(k)$
\begin{eqnarray}
\label{eq:cellular}
  \mathrm{U}(Y_i) \to \mathrm{U}(X\backslash X_{i-1}) \to \mathrm{U}(X\backslash X_i) \stackrel{\partial}{\too} 
  \Sigma U(Y_i)&& 0 \leq i \leq n-1\,.
\end{eqnarray}
We will prove by descending of induction on $j$ for $-1\leq j\leq n-1$
that $\mathrm{U}(X\backslash X_j) \simeq
\bigoplus_{i=j+1}^n\mathrm{U}(Y_i)$.  For $j=n-1$ this boils down
to $\mathrm{U}(X\backslash X_{n-1})=\mathrm{U}(Y_n)$ which has already been proved. Assume now that $\mathrm{U}(X\backslash X_j)\simeq \bigoplus_{i=j+1}^n \mathrm{U}(Y_i)$. Then, the distinguished Gysin
triangle \eqref{eq:cellular} for $i=j$ becomes
\begin{equation}
\label{eq:split}
\mathrm{U}(Y_j) \too \mathrm{U}(X\backslash X_{j-1}) \to  \bigoplus_{i=j+1}^n \mathrm{U}(Y_i) \stackrel{\partial}{\too} \Sigma \mathrm{U}(Y_j)\,.
\end{equation}
Since the schemes $Y_i$'s are smooth projective, Proposition \ref{prop:K-theory1} implies 
that $\partial=0$. Therefore, the distinguished triangle \eqref{eq:split} splits and gives rise to an isomorphism $\mathrm{U}(X\backslash X_{j-1})\simeq \bigoplus_{i=j}^n \mathrm{U}(Y_i)$. This completes the proof of the induction step.

\medskip

According  to Lemma \ref{lem:invertible}, the additive functor $\overline{\mathrm{U}}$ (see \S\ref{sec:universal}) restricts to an equivalence
  between the full subcategories of $\Hmo_0(k)$ and $\Mot(k)$ spanned by
the  objects $\mathrm{U}_{\mathrm{add}}(X)$ and $\mathrm{U}(X)$, respectively, where $X$
  runs through the smooth proper $k$-schemes. As a consequence, we have also an induced isomorphism $\mathrm{U}_{\mathrm{add}}(X)\simeq \bigoplus_{i=0}^n \mathrm{U}_{\mathrm{add}}(Y_i)$ in the additive category $\Hmo_0(k)$. The proof follows now from the fact that $F$, being additive, factors through $\mathrm{U}_{\mathrm{add}}$.
%-------------------------------------------------------------------------------
%\section*{Proof of Theorem \ref{thm:new2}}
%-------------------------------------------------------------------------------

%\begin{proposition}\label{prop:generators}
%Let $k$ be a perfect field of characteristic $p>0$, $l$ a prime number different from $p$, and $X$ a smooth $k$-scheme. The object $\Sigma^\infty(X_+)_{\Michel{\bbZ[1/p]}}\wedge \mathrm{KGL}_{\Michel{\bbZ[1/p]}}$ belongs to the smallest thick triangulated subcategory of $\Mod(\mathrm{KGL}_{\Michel{\bbZ[1/p]}})$ containing the objects $\Sigma^\infty(Y_+)_{\Michel{\bbZ[1/p]}} \wedge \mathrm{KGL}_{\Michel{\bbZ[1/p]}}$ with $Y$ a smooth projective scheme.
%\end{proposition}
%\begin{proof}
%The proof is similar to the proof of Theorem \ref{thm:main2}: simply replace the Gysin triangles \eqref{eq:triangle} by the motivic Gysin triangles (see \S\ref{sub:relation})
%$$
%\Sigma^\infty(U_+)\wedge \mathrm{KGL} \too \Sigma^\infty(X_+)\wedge \mathrm{KGL} \too \Sigma^\infty(Z_+)\wedge \mathrm{KGL}\stackrel{\partial}{\too} \Sigma(\Sigma^\infty(U_+)\wedge \mathrm{KGL})\,.
%$$
%%and \eqref{eq:composition} by the following composition
%%$$ \Sigma^\infty(Y_+)\wedge \mathrm{KGL} \too \Sigma^\infty(X_+)\wedge \mathrm{KGL} \stackrel{\Sigma^\infty(f_+)\wedge \mathrm{KGL}}{\too} \Sigma^\infty(Y_+)\wedge \mathrm{KGL}\,,$$
%%where
%\end{proof}

\end{document}

\end{proof}